\documentclass[12pt]{amsart}

\usepackage{amsmath, amsthm, amssymb}
\usepackage{tikz}
\usepackage{graphicx}
\usepackage{amsfonts}
\usepackage{pgfplots}
\usepackage{comment}
\usepackage{esint}
\usepackage{cancel}
\usepackage{eucal}
\usepackage{enumerate}
\usepackage{xcolor}
\usepackage[colorlinks,linkcolor=blue]{hyperref}

\allowdisplaybreaks[4]

\numberwithin{equation}{section}

\newtheorem{thm}{Theorem}[section]
\newcommand{\bt}{\begin{thm}}
\newcommand{\et}{\end{thm}}

\newtheorem{cor}[thm]{Corollary}   
\newcommand{\bc}{\begin{cor}}
\newcommand{\ec}{\end{cor}}

\newtheorem{lem}[thm]{Lemma}   
\newcommand{\bl}{\begin{lem}}
\newcommand{\el}{\end{lem}}

\newtheorem{prop}[thm]{Proposition}
\newcommand{\bp}{\begin{prop}}
\newcommand{\ep}{\end{prop}}

\newtheorem{defn}[thm]{Definition}
\newcommand{\bd}{\begin{defn}}    
\newcommand{\ed}{\end{defn}}

\newtheorem{rmrk}[thm]{Remark}   

\newcommand{\br}{\begin{rmrk}}
\newcommand{\er}{\end{rmrk}}

\newcommand{\be}{\begin{equation}}
\newcommand{\ee}{\end{equation}}

\newcommand{\N}{\mathbb{N}}

\newcommand{\R}{\mathbb{R}}

\newcommand{\Z}{\mathbb{Z}}

\newcommand{\g}{\overline{g}}
\newcommand{\e}{\overline{e}}
\newcommand{\on}{\overline{\nabla}}
\newcommand{\pr}{\partial_r}
\newcommand{\D}{\overline{D}}
\newcommand{\vol}{{\rm vol}}
\newcommand{\Scal}{{\rm Scal}}
\newcommand{\tu}{\tilde{u}}
\newcommand{\tv}{\tilde{v}}

\begin{document}

\title[PMT for AF manifolds with isolated conical singularities]{Positive mass theorem for asymptotically flat spin manifolds with isolated conical singularities}

\author{Xianzhe Dai}
\address{
Department of Mathematics, 
University of Californai, Santa Barbara
CA93106, USA}
\email{dai@math.ucsb.edu}

\author{Yukai Sun}
\address{Key Laboratory of Pure and Applied Mathematics, 
School of Mathematical Sciences, Peking University, Beijing, 100871, P. R. China
}
\email{sunyukai@math.pku.edu.cn}

\author{Changliang Wang}
\address{
School of Mathematical Sciences and Institute for Advanced Study, Tongji University, Shanghai 200092, China}
\email{wangchl@tongji.edu.cn}

\date{}


\begin{abstract}
There has been a lot of interests in Positive Mass Theorems for singular metrics on smooth manifolds. We prove a positive mass theorem for asymptotically flat (AF) spin manifolds with isolated conical singularities or more generally horn singularities.
In particular, we allow topological singularities in the space as we do not require the cross sections of the conical singularity to be spherical. Note that the negative mass Schwarzschild metric is AF with a horn singularity.
\end{abstract}

\maketitle

\tableofcontents
\section{Introduction}

The famous Positive Mass Theorem \cite{SY-PMT, Witten-PET} states that an asymptotically flat (AF for short; also called asymptotically Euclidean) manifold with nonnegative scalar curvature must have nonnegative ADM mass (if the dimension of the manifold is between $3$ and $7$ or if the manifold is spin; for the recent progress about the higher dimensional non-spin manifolds, see \cite{SY1}). Furthermore, the mass is zero iff the manifold is the Euclidean space. The remarkable applications
of this seminal result include the Geroch conjecture (and various scalar curvature rigidity results), Schoen's final resolution of the Yamabe conjecture, Bray's proof of the Penrose inequality, and the so-called Black Hole uniqueness results.

There has been a lot of interests in Positive Mass Theorems for singular metrics on smooth manifolds (also referred as Positive Mass Theorems with low regularity). Part of the motivations come from attempts in removing the spin restriction in the higher dimensions and the study of stability aspect of Positive Mass Theorems. 
Another source of motivation comes from Gromov's polyhedral comparison theory for positive scalar curvature \cite{Gromov14} and investigating weak notions of nonnegative scalar curvature. These are metric singularities occurring on smooth manifolds, and the metric is usually assumed to be continuous, see e.g. \cite{Grant-Tassotti, JSZ2022, Lee-PAMS-2013, Lee-LeFloch, Li-Mantoulidis, Miao2002,Shi-Tam-PJM-2018} and others. In \cite{Grant-Tassotti}, Grant-Tassotti proved a positive mass theorem for metrics with $W^{2, \frac{n}{2}}_{loc}$ regularity for dimension $n\leq 7$ or spin manifold of any dimensions. In \cite{Lee-PAMS-2013}, Lee proved a positive mass theorem for $W^{1, p}_{loc}$-metrics ($n<p\leq \infty$) whose singular sets have zero $\frac{n}{2}\left( 1- \frac{n}{p} \right)$-dimensional lower Minkowski content, for $n\leq 7$ or spin manifolds of any dimension. In \cite{Shi-Tam-PJM-2018}, Shi and Tam proved a positive mass theorem for $W^{1, p}_{loc}$-metrics ($n<p \leq \infty$) whose singular sets have codimensions at least $2$, for $n\leq 7$.  In \cite{Lee-LeFloch}, Lee and LeFloch proved a positive mass theorem for spin manifolds with $W^{1, n}_{loc}$-metrics, without imposing any constraint on the size or dimension of the singular sets of the metrics. In the non-spin case, Jiang, Sheng and Zhang \cite{JSZ2022} proved a positive mass theorem for $W^{1, p}_{loc}$-metrics $(n \leq p \leq \infty)$ whose singular sets $\Sigma$ have finite Hausdorff measure $\mathcal{H}^{n-\frac{p}{p-1}}(\Sigma)<\infty$ if $n\leq p<\infty$ or $\mathcal{H}^{n-1}(\Sigma)=0$ if $p=\infty$.

In this paper we study the Positive Mass Theorems in the presence of isolated conical or more generally horn singularities.
In particular, we allow topological singularities in the space as we do not require the cross sections of the conical singularity to be spherical. Even if the cross sections are diffeomorphic to sphere and hence the manifolds are smooth, our asymptotically conical metrics (as in Definition \ref{defn-conic-mfld}) will still not satisfy the regularity assumptions of the above works on the metrics, if the metric on cross section is not the standard spherical metric. 

The motivation of our work is to understand the extent singularity may affect positive mass theorems, which is connected with the Schoen conjecture about co-dimension $3$ singularity. The conjecture says that on a closed manifold with nonpositive Yamabe constant (aka $\sigma$-constant or Schoen constant), a continuous uniformly Euclidean metric with co-dimension $3$ (or higher) singularity and nonnegative scalar curvature on the smooth part must already be smooth, hence Ricci flat \cite{Li-Mantoulidis}. Here uniformly Euclidean metrics are those which are quasi-isometric to smooth metrics. On a smooth manifold of dimension $3$ or higher, a smooth metric with only conical singularities are uniformly Euclidean. The Schoen-Yau-Lohkamp compactification scheme turns an AF manifold into a closed manifold with nonpositive Yamabe constant. In this scheme, a negative mass metric with singularity and nonnegative scalar curvature on the AF manifold should correspond to a metric with singularity and nonnegative scalar curvature on the compactification. Thus, if Schoen conjecture is true, one would expect a positive mass theorem for continuous metrics with conical singularity. Of course, we show that this is the case even when there is topological singularity.

Conical singularities occur naturally in the study of the near horizon geometry of black holes. Motivated from this consideration, a positive mass theorem for conical singularity under the nonnegative Ricci curvature is established in \cite{Lucietti-21} (modulo some analytic details). 

The isolated conical singularities and horn singularities are included in the class of the so-called zero area singularities studied in \cite{Bray-Jauregui-AJM-2013}. Assuming a conformal conjecture (see Conjecture 34 in \cite{Bray-Jauregui-AJM-2013}) to be true, Bray and Jauregui \cite{Bray-Jauregui-AJM-2013} obtained a lower bound for the ADM mass at infinity in term of the mass of singularities that they defined. In particular, if their conformal conjecture holds and the mass of singularities is nonnegative, then a positive mass theorem would follow. However, to our knowledge, Conjecture 34 in \cite{Bray-Jauregui-AJM-2013} is still open. And the mass of (even) conical singularities does not seem easy to estimate.

In the next section we will define precisely what we mean by asymptotically flat (AF) manifolds with (finitely many) isolated conical singularities. Our main result is

\begin{thm}
Let $(M^n, g)$ be a AF spin manifold with finitely many isolated conical singularity and $n\geq 3$. If the scalar curvature is nonnegative on the smooth part, then the mass $m(g)$ is nonnegative. Furthermore, the mass $m(g)=0$ if and only if $(M^{n},g)\simeq (\mathbb{R}^{n},g_{\mathbb{R}^{n}})$, except when $n=4k$, $k\in \mathbb{Z}_{+}$, we require additionally that the cross sections of the model cones of the conical singularities are simply connected.
\end{thm}

This result is proved in Theorems \ref{thm-PMT}. Our proof also works for multiple AF ends (or even some ALE ends but the spin condition is crucial). In the exceptional dimensions $n=4k$, $k\in \mathbb{Z}_{+}$, the rigidity result still holds if the fundamental groups of the cross sections of the conical singularities are not $\mathbb Z_2$. This is related to B\"ar's classification \cite{Bar-JMSP-1996} of the spaces with maximum number of Killing spinors, as discussed below. 

In 3-dimension, by estimating capacity and Willmore functional of the cross sections of a cone sufficiently close to the cone tip, Theorem 7.4 in \cite{Miao2023} implies the nonnegativity of the mass, provided that the relative homology group $H_2(M\setminus B_{r_0}(o), \partial B_{r_0}(o)) = 0$, where $r_0$ is a sufficiently small positive number and $B_{r_0}(o)$ is the ball centered at conically singular point $o$ with radius $r_0$.
In \cite{TV2023}, Ju and Viaclovsky proved a positive mass theorem for AF manifolds with isolated orbifold singularities. 
In \cite{Li-Mantoulidis}, Li and Mantoulidis proved a positive mass theorem for metrics with conical singlarities along a codimension two submanifold (aka edge metrics) on smooth AF manifolds. In \cite{Shi-Tam-PJM-2018}, Shi-Tam studied some examples of AF manifolds with conical and horn singularities whose cross sections are constant scaling of standard round spheres, and calculated their ADM mass at infinity.

The negative mass Schwarzschild metric $g=(1-\frac{2m}{r})^4 g_{\mathbb R^3}, m>0$ defined on $r>2m$ in ${\mathbb R^3}$ is scalar flat and has ADM mass $-m$. It was observed in \cite[Section 4.2]{Bray-Jauregui-AJM-2013} (see also Proposition 2.3 in \cite{Shi-Tam-PJM-2018}) that near $r=2m$, 
$$ g=d\sigma^2 + c\sigma^{4/3} (1+O(\sigma^{2/3}))h_0,
$$
where $h_0$ is the standard metric on $\mathbb{S}^2$. We call such type of singularity as horn singularity. 
Thus the negative mass Schwarzschild metric is AF with a $r^{2/3}$-horn singularity, scalar flat but has negative mass.

On the other hand, we note, as a consequence of Herzlich's Positive Mass Theorem for AF manifolds with boundary, 

\begin{thm}\label{thm-PMT-horns-intro}
Let $(M^n, g)$, $n\geq 3$, be an AF spin manifold with finitely many $r^b$-horn singularities $($$b$ may be different for each horn$)$. Assume $b>1$ for each horn, and the Yamabe invariant of the cross section of 
each singularity is strictly positive. If the scalar curvature of $g$ is nonnegative, then the mass is strictly positive.
\end{thm}

\begin{rmrk}
{\rm
   In Theorem \ref{thm-PMT-horns-intro}, the positive Yamabe invariant assumption for the cross section is used to derive the assumption $(\ref{eqn-horn-Herzlich-condition})$ below in Herzlich's result. But because nonnegativity of scalar curvature of an exact horn metric implies the positivity of the Yamabe invariant of the cross section, if the metric near the singularity is an exact horn metric, then we obtain a positive mass theorem without the positive Yamabe invariant assumption, see Corollary \ref{cor-PMT-exact-horn}. 
}
\end{rmrk}

Our general approach follows that of Witten. Thus we must analyze the Dirac operator on AF manifolds with conical singularity, solve the Dirac equation,  and study the asymptotic behaviors of its solutions. To deal with the conical singularity, we introduce weighted Sobolev spaces and study the mapping properties of the Dirac operator, especially its Fredholm property. Our treatment is adapted from that of \cite{Minerbe-CMP}, where the case of Euclidean space is dealt with. In fact our weighted Sobolev spaces come with two weights, one for the conical singularity and the other for the spatial infinity. When the weights are not critical, i.e., not equal to the indicial roots, we establish a refined elliptic estimate, which is crucial for proving the Fredholm property. With the Fredholm property, the surjectivity of the Dirac operator as a map between weighted Sobolev spaces can be obtained from the nonnegative scalar curvature assumption and the Lichnerowicz formula, provided the weights are constrained. One novelty here is that the weight constraints can actually be improved, which turns out to be critical in solving the Dirac equation to obtain harmonic spinors with certain asymptotic control near the conical point and asymptotic to constant spinors at infinity. This is essential in the proof of the positive mass theorem.

Conical singularity presents new difficulty in proving the rigidity case of the positive mass theorem as well. As usual zero mass leads to non trivial parallel spinors, and hence the space must be Ricci flat. However, the presence of singularity prevents the easy application of the rigidity case of the classical relative volume comparison (one may still be able to treat it as an RCD space though there are subtleties). To overcome this, we prove a partial asymptotic expansion for parallel spinors at the conical singularity, with the leading order term having nonnegative vanishing order determined by the eigenvalues of the Dirac operator on the cross section of the model cone. Because a non-trivial parallel spinor cannot have a positive vanishing order, the vanishing order must be zero. Then the rigidity result of Friedrichs \cite{Friedrich-80} let us to conclude that the leading term of the partial expansion is in fact a Killing spinor. Finally we make use of a result of B\"ar \cite{Bar-JMSP-1996} characterizing the spaces having maximal number of Killing spinors. Note that in certain dimensions one has the real projective spaces as well as the standard spheres, which is the reason for our additional assumption in those dimensions. 


The paper is organized as follows. In Section \ref{sec-conic}, we introduce asymptotically flat (AF) manifolds with isolated conical singularities and their ADM mass. Section \ref{sec-basic} collects some basic facts on cones, including the Dirac operators on the cones. Section \ref{sec-anal} is devoted to the analysis on AF manifolds with isolated conical singularities. We introduce the weighted Sobolev spaces and discuss the corresponding elliptic estimates for the Dirac operators, their Fredholm properties, and surjectivity. This is used in Section \ref{sec-solv} to obtain asymptotically constant harmonic spinors with also asymptotic control at the conical singularities. In Section \ref{sec-comp} we prove the positive mass theorem for conical singularity (Theorem 1.1). Finally Section \ref{horn} treats the horn singularity.

{\em Acknowledgement}: The first author would like to acknowledge several very interesting conversations with James Lucietti and Hari Kunduri while attending the Banff workshop in April, 2023. The authors thank Pengzi Miao for informing us that his result \cite{Miao2023} implies nonnegativity of the mass in 3-dimension and for useful discussions. We also thank Jeff Viaclovsky for bring their work \cite{TV2023} to our attention and for valuable comments.

\section{AF manifolds with isolated conical singularities} \label{sec-conic}
In this section, we give the precise definition of what we call asymptotically flat manifolds with isolated conical singularities, and the definition of the mass (at infinity) for these singular manifolds.
\begin{defn}\label{defn-conic-mfld}
{\rm
We say $(M^n_0, g, d, o)$ is a compact Riemannian manifold with smooth boundary and a single conical singularity at $o \in M_0 \setminus \partial M_0$, if 
\begin{enumerate}[(\romannumeral1)]
\item $d$ is a metric on $M_0$ and $(M_0, d)$ is a compact metric space with smooth boundary,
\item $g$ is a smooth Riemannian metric on the regular part $ M_0 \setminus \{ o \}$, 
         $d$ is the induced metric by the Riemannian metric $g$ on $M_0 \setminus \{ o \}$ ,
\item there exists a neighborhood $U_o$ of $o$ in $M \setminus \partial M$, such that 
         $U_o \setminus \{ o \} \simeq (0, 1) \times N$ 
          for a smooth compact manifold $N$, 
           and on $U_o \setminus \{ o \}$ the metric $g = \g + h $, where
         \[
         \g = dr^2 + r^2 g^N,
         \]
         $g^N$ is a smooth Riemannian metric on $N$, $r$ is a coordinate on $(0, 1)$, $r=0$ corresponding the singular point $o$, 
         and $h$ satisfies 
         \[
         |\on^k h|_{\g} = O(r^{ \alpha - k }),  \ \ \text{as} \ \ r \rightarrow 0,
         \]
         for some $\alpha >0$ and $k = 0, 1 $ and $2$, where $\on$ is the Levi-Civita connection of $\g$.
\end{enumerate} 
}
\end{defn}

\begin{defn}\label{defn-AF-conic-mfld}
{\rm
We say $(M^n, g, o)$ is an asymptotically flat manifold with a single isolated conical singularity at $o$, if $M^n = M_0 \cup M_\infty$ satisfies
\begin{enumerate}[(\romannumeral1)]
\item $( M_0, g|_{M_0 \setminus\{o\}}, o)$ is a compact Riemannian manifold with smooth boundary and a single conical singularity at $o$ 
         defined as in Definition \ref{defn-conic-mfld}, 
\item  $M_\infty \simeq  \R^n \setminus B_{R}(0) $ for some $R > 0$, and the smooth Riemannian metric $g$ on $M_\infty$ satisfies
         \[
          g = g_{\R^n} + O(\rho^{-\tau}), \ \ |(\nabla^{g_{\R^n}})^i g|_{g_{\R^n}} = O(\rho^{-\tau - i}),  \ \  \text{as} \ \ \rho \rightarrow +\infty,
         \]
         for $i = 1, 2$ and $3$,
         where $\tau > \frac{k-2}{2}$ is the asymptotical order, $\nabla^{g_{\R^n}}$ is the Levi-Civita connection of the Euclidean metric $g_{\R^n}$, and $\rho$ is the Euclidean 
         distance to a base point.
\end{enumerate}
}
\end{defn}

\begin{rmrk}
{\rm
For the sake of simplicity of notations, in Definition \ref{defn-AF-conic-mfld} we only defined AF manifolds with a single conical singularity and a single AF end, and we will only focus on this case in this paper. But AF manifolds with finitely many isolated conical singularities as well as finitely many AF ends can be defined similarly, and all works in this paper can be easily extended to the case of finitely many isolated conical singularities and multiple AF ends.
}
\end{rmrk}

Now we give the definition of the ADM mass of such Riemannian manifold.
\begin{defn}\label{def-mass}
{\rm
Let $(M^n, g, o)$ be an asymptotically flat manifold with a single isolated conical singularity at $o$. The mass $m(g)$ is defined 
 \[m(g)=\lim_{R\to\infty}\frac{1}{\omega_{n}}\int_{S_{R}}(\partial_{i} g_{ji}-\partial_{j}g_{ii})\ast dx^{j},\]
 where $\{\frac{\partial}{\partial x^{i}}\}$ is an orthonormal basis of $g_{\mathbb{R}^{n}}$ and the $\ast$ operator is the Hodge star operator on the Euclidean space, the indices $i,j$ run over $M^n$ and $S_{R}$ is the sphere of radius $R$ on $\mathbb{R}^{n}$ and $\omega_{n}$ is the volume of the unit sphere in $\mathbb{R}^{n}$.
 }
\end{defn}

\section{Basic facts on cones}\label{sect-Dirac-op-square-cone} \label{sec-basic}
In this section, we recall some basic facts about geometry on cones, and then derive the relation between the square of Dirac operator on a cone and on its cross section. This will be useful in \S$\ref{sect-analysis}$.

\subsection{Levi-Civita connection and curvature tensor on cones}

Let $(N^{n-1}, g^N)$ be a compact Riemannian manifold, and $(C(N), \g) = (\R_+ \times N, dr^2 + r^2 g^N)$ be the Riemannian cone over $(N^{n-1}, g^N)$, where $r$ is the coordinate on $\R_+$. Let $\{ e_1, \cdots, e_{n-1} \}$ be a local orthonormal frame of $TN$ with respect to $g^N$, $\e_i = \frac{1}{r} e_i$ for all $i=1, \cdots, n-1$, and $\partial_r = \frac{\partial}{\partial r}$. Then $\{\e_1, \cdots, \e_{n-1}, \partial_r\}$ is a local orthonormal frame of $TC(N)$ with respect to $\g$. Let $\on$ denote the Levi-Civita connection of $\g$. By Koszul's formula, one easily obtains
\be\label{eqn-connection-cone}
\begin{cases}
 & \on_{\e_i} \pr = \frac{1}{r} \e_i, \\
 & \on_{\pr} \e_i = [ \pr, \e_i ] + \on_{\e_i} \pr = - \frac{1}{r} \e_i + \frac{1}{r} \e_i =0, \\
 & \on_{\pr} \pr = 0, \\
 & \on_{\e_i} \e_j = \nabla^{g^N}_{\e_i} \e_j  - \frac{1}{r} \delta_{ij} \pr,
\end{cases}
\ \ \forall 1 \leq i, j \leq n-1.
\ee
Let $\overline{R}$ denote the Riemann curvature tensor of $\g$, and $R^N$ denote the Riemann curvature tensor of $g^N$ on $N$. In the local frame $\{\e_1, \cdots, \e_{n-1}, \partial_r\}$, the only non-vanishing components of $\overline{R}$ are
\be
\overline{R}_{ijkl} = r^2 \left( R^N_{ijkl} + g^N_{ik} g^N_{jl} - g^N_{il} g^N_{jk} \right), \ \ 1 \leq i, j, k, l \leq n-1.
\ee
In particular,
\be\label{eqn-curvature-cone}
\overline{R} (\pr, X, Y, Z) =0, \ \ \forall X, Y, Z \in \Gamma(TC(N)).
\ee
For the hypersurface $N_r := \{ r \} \times N$, by the covariant derivative formula in (\ref{eqn-connection-cone}), the shape operator $A$ with respect to the normal vector field $\pr$ is given by
\be\label{eqn-shape-operator-cone}
A X :=   \nabla_X \pr =  \frac{1}{r} X, \quad \forall X \in \Gamma(TN_r).
\ee
Then the mean curvature of $N_r$ with respect to the normal vector field $\partial_r$ is given by
\be\label{eqn-mean-curvature-cone}
H :=  {\rm tr}_{\g} A = \frac{n-1}{r}.
\ee


\subsection{Dirac operator on cones}
If $N$ is spin, then so is $C(N)$. In this case, we assume that the spin structure of $C(N)$ is induced by that of $N$. Under this assumption,  the restriction of the spinor bundle $SC(N)$ over the Riemannian cone $(C(N), \g)$ to the hypersurface $N_r := \{r\} \times N$ is isomorphic to either $S N$ or $SN \oplus SN$ according to the dimension $n-1$ of $N$ is either even or odd. Then we use the same symbol to denote a spinor field on $C(N)$ and its restriction to $N_r$. Let $\D$ denote the Dirac operator on $(C(N), \g)$, and $D^{r^2 g^N}$ the Dirac operator on $(N_r, r^2 g^N)$.  We recall the relation between these two Dirac operators given in the equation (9) in \cite{HMZ-MRL} as
\be\label{eqn-Dirac-op-relation-1}
D^{r^2 g^N} \varphi  =- \frac{1}{2} H \varphi - \pr \cdot \D \varphi - \on_{\pr} \varphi,
\ee
where $H$ is the mean curvature of $N_r$ given as in $(\ref{eqn-mean-curvature-cone})$. {\em Note the difference between our convention of mean curvature and that in \cite{HMZ-MRL}. Here and throughout the paper, the `` $\cdot$ " between a vector field and  a spinor field always denotes the Clifford multiplication on the cone $(C(N), \g)$.}  Then by plugging the mean curvature given in (\ref{eqn-mean-curvature-cone}) into the formula $(\ref{eqn-Dirac-op-relation-1})$, we have
\be
D^{r^2 g^N} \varphi  = -\frac{n-1}{2r}  \varphi - \pr \cdot \D \varphi - \on_{\pr} \varphi = -\frac{n-1}{2r}  \varphi - \pr \cdot \D \varphi - \pr( \varphi ).
\ee
Here we used the fact $\on_{\pr} \varphi = \pr (\varphi)$ followed from the covariant derivatives on the cone given in (\ref{eqn-connection-cone}). Rearranging the equation and Clifford acting by $\partial_r$ give
\begin{equation}\label{eqn-Dirac-op-relation}
\overline{D} \varphi = \partial_r \cdot \partial_r(\varphi) + \frac{n-1}{2r} \partial_r \cdot \varphi + \frac{1}{r} \partial_r \cdot D^{g^N} \varphi.
\end{equation}

Moreover, under the above identification of spinors on a cone and that on its cross section, the Clifford action of the vector field $\partial_r$ on spinors anti-commutes with the Dirac operator $D^{g^N}$, i.e. \begin{equation}\label{eqn-partial-r-D-anti-commute}
\partial_r \cdot \left( D^{g^N} \varphi \right) = - D^{g^N} \left( \partial_r \cdot \varphi \right)
\end{equation}
for all spinors $\varphi$ on $(N, g^N)$, see Proposition 1 in \cite{HMZ-MRL}.


\section{Analysis on AF manifolds with isolated conical singularities}\label{sect-analysis} \label{sec-anal}

In this section, we mainly study analytical property of  the Dirac operator $D$. In \S$\ref{subsect-weighted-Sobolev-space}$ we introduce the weighted Sobolev spaces that we use. Then we derive a priori elliptic estimate for $D$ in Proposition $\ref{prop-refined-weighted-elliptic-estimate}$, and as a result we obtain the Fredholm property for $D$ in \S$\ref{subsect-Fredholm}$. With the help of this Fredholm property, we then show the surjectivity of the map $D$ between certain weighted Sobolev spaces in Proposition $\ref{prop-surjectivity-1}$. This enables us to solve the inhomogeneous equation $D \varphi = \psi$, and will help us to solve the homogeneous Dirac equation in \S$\ref{sect-sovlving-Dirac-equation}$.

\subsection{Weighted Sobolev spaces}\label{subsect-weighted-Sobolev-space}
Let $(M^n, g, o)$ be a AF manifold with a single conical singularity at $o$ as defined in Definition \ref{defn-AF-conic-mfld}. Choose three cut-off functions $0 \leq \chi_1, \chi_2, \chi_3 \leq 1$ satisfying:
\be
\chi_1 (x) =
\begin{cases}
1, &  {\rm dist}(x, o) < \epsilon, \\
0, & {\rm dist}(x, o) > 2 \epsilon,
\end{cases}
\ee
\be
\chi_2(x) = 
\begin{cases}
1, & {\rm dist}(x, o) > 2R, \\
0, & {\rm dist}(x, o) < R,
\end{cases}
\ee
and
\be
\chi_3 = 1 - \chi_1 - \chi_2,
\ee
where $\epsilon >0$ is chosen sufficiently small such that the ball $B_{2\epsilon}(o)$, centered at singular point $o$ with radius $2\epsilon$, is contained in the asymptotically conical neighborhood $U_o$ in Definition \ref{defn-conic-mfld}, and $R>0$ is chosen sufficiently large such that $M \setminus B_{R}(o) \subset M_\infty$. 

For each $1 \leq p < +\infty, k \in \mathbb{N} $ and $\delta, \beta \in \R$,  the {\em weighted Sobolev space} $W^{k, p}_{\delta, \beta}(M)$ is defined to be the completion of $C^\infty_0(M\setminus \{o\})$ with respect to the {\em  weighted Sobolev norm } given by
\be
\| u \|^p_{W^{k, p}_{\delta, \beta}(M)} :=  \int_{M} \sum^{k}_{i=0} \left( r^{- p(\delta - i) -n} |\nabla^i u |^p \chi_1 + \rho^{ - p( \beta - i) -n} |\nabla^i u |^p \chi_2 + |\nabla^i u|^p \chi_3 \right) d\vol_{g} ,
\ee
where $r$ is the radial coordinate on the conical neighborhood $U_o$ in Definition \ref{defn-conic-mfld} and $\rho$ is the Euclidean distance function to a base point on $M_\infty$ in Definition \ref{defn-AF-conic-mfld}, and $|\nabla^i u|$ is the norm of $i^{th}$ covariant derivative of $u$ with respect to $g$.

Note that by the definition of the weighted Sobolev norms, we clearly have
\be
\begin{cases} \delta^\prime \ge \delta, \cr \beta^\prime \leq \beta, \end{cases}
\Rightarrow \ \  W^{k, p}_{\delta^\prime, \beta^\prime} \subset W^{k, p}_{\delta, \beta}.
\ee


If $V$ is a real (or complex) vector bundle over $M \setminus \{ o \}$ with a Euclidean (or Hermitian) metric and a compatible connection, then we can define the weighted sobolev norms and spaces $W^{k, p}_{\delta, \beta}(V)$, and weighted $C^k$ norms and spaces $C^k_{\delta, \beta}(V)$ for sections of the vector bundle $V$ in the same way as we define $W^{k, p}_{\delta, \beta}(M)$ and $C^k_{\delta, \beta}(M)$ for functions. In this paper, the vector bundle is usually the spinor bundle, which is denoted as $SM$.

\subsection{Elliptic estimate for the Dirac operator}

First we prove that nonnegativity of scalar curvature near the conical point implies a uniform lower bound for the absolute value of eigenvalues of the Dirac operator on the cross section $N$ of conical neighborhood of the singularity.
\begin{lem}\label{lem-cross-section-eigenvalue-lower-bound}
Let $(M^n, g, o)$, $n \geq 3$, be an AF spin manifold with a single conical singularity at $o$ whose scalar curvature $\Scal_g \geq 0$. Then the eigenvalues $ \lambda_j$ of Dirac operator $D^{g^N}$ on the cross section of a conical neighborhood of $o$ satisfy 
\be
|\lambda_j| \geq \frac{n-1}{2},  \ \ \forall j \in \Z.
\ee
The equality holds if and only if the cross section $(N, g^N)$ admits a real Killing spinor.
\end{lem}
\begin{proof}
By straightforward calculations, one obtains that near the conical singularity $o$, the scalar curvature of the metric $g$ has the asymptotic behavior as
\be
{\rm Scal}_{g} = \frac{1}{r^2} \left( {\rm Scal}_{g^N} - (n-1)(n-2) + O(r^\alpha) \right), \ \ \text{as} \ \ r\rightarrow 0.
\ee
As a result, ${\rm Scal}_g \geq 0$ implies ${\rm Scal}_{g^N} \geq (n-1)(n-2)$. Then by combining this with the eigenvalue estimate for Dirac operator obtained by Thomas Friedrich in \cite{Friedrich-80}, we have that the eigenvalues $\lambda_j$ of Dirac operator $D^{g^N}$ satisfy
\be
|\lambda_j| \geq \frac{1}{2} \sqrt{\frac{n-1}{n-2} \min_{N}\{{\rm Scal}_{g^N}\}} \geq \frac{n-1}{2},
\ee
and the lower bound realized by Riemannian manifolds admitting real Killing spinors.

\end{proof}

Throughout the paper, we will always use $\lambda_j$ to denote the eigenvalues of the Dirac operator $D^{g^N}$ on the cross section of the model cone $(C(N), \overline{g})$, and $E_j$ to denote the space of the corresponding eigenspinors. We set 
\begin{equation}\label{eqn-nu-defn}
\nu_j := - \frac{n-1}{2} - \lambda_j, \ \ j \in \mathbb{Z}.
\end{equation}

\begin{defn}\label{defn-critical-cone}
{\rm
We will say $\delta \in \R$ is {\em critical at conical point } if $\delta = \nu_j$ for some $j \in \mathbb{Z}$, where $\nu_j$ is defined as in (\ref{eqn-nu-defn}).
}
\end{defn}

We view the Euclidean space $(\R^n, g_{\R^n})$ as a cone over the round sphere $\mathbb{S}^{n-1}$ with constant sectional curvature $1$. In Theorem 1 in \cite{Bar-JMSP-1996}, C. B\"ar proved that the Dirac operator on the round sphere $\mathbb{S}^{n-1}$ has the eigenvalues $\pm \left( \frac{n-1}{2} +k \right)$ for nonnegative integers $ k \in \Z_{\geq 0}$. By plugging these eigenvalues into the expression of critical indices in Definition \ref{defn-critical-cone}, we give the following definition.

\begin{defn}\label{defn-critical-infinity}
{\rm
We will say $\beta \in \R$ is {\em critical at infinity } if $\beta = k $ or $\beta = 1-n-k$ for some $ k \in \Z_{\geq 0} $.
}
\end{defn}

\begin{defn}\label{defn-critical-index}
{\rm
We say $(\delta, \beta) \in \R^2$ is {\em critical} if either $\delta$ is critical at conical point or $\beta$ is critical at infinity.
}
\end{defn}

By using scaling technique, the usual interior elliptic estimates, and the asymptotic control of metric $g$ near conical point in Definition $\ref{defn-conic-mfld}$ and near infinity in Definition $\ref{defn-AF-conic-mfld}$, we have the following weighted elliptic estimate.
\begin{prop}\label{prop-weight-elliptic-estimate}
For $k=1, 2$, and $p>1$, if $\varphi \in L^{p}_{\delta, \beta}(SM)$, and $D \varphi \in W^{k-1, p}_{\delta-1, \beta-1}(SM)$, then
\be
\|\varphi \|_{W^{k, p}_{\delta, \beta}(SM)} \leq C \left( \|D \varphi \|_{W^{k-1, p}_{\delta-1, \beta-1}(SM)} + \|\varphi\|_{L^{ p}_{\delta, \beta}(SM)} \right)
\ee
holds for some constant $C = C(g, n, k)$ independent of $\varphi$.
\end{prop}
Note that Proposition $\ref{prop-weight-elliptic-estimate}$ is true for any indices $(\delta, \beta) \in \mathbb{R}^2$. For noncritical indices $(\delta, \beta) \in \mathbb{R}^2$, we further obtain the following refined elliptic estimate, which plays a crucial role in \S $\ref{subsect-Fredholm}$.

\begin{prop}\label{prop-refined-weighted-elliptic-estimate}
If $(\delta, \beta) \in \R^2$ is not critical as in Definition $\ref{defn-critical-index}$, for $k=1, 2$, there exists a constant $C= C(g, n, k)$ and a compact set $B \subset M \setminus \{o\}$ such that for any $\varphi \in L^2_{\delta, \beta}(SM)$ with $D \varphi \in L^{2}_{\delta, \beta}(SM)$,
\be
\|\varphi\|_{W^{k, 2}_{\delta, \beta}(SM)} \leq C \left( \|D \varphi \|_{W^{k-1, 2}_{\delta-1, \beta-1}(SM)} + \|\varphi\|_{L^{2}(SB)} \right).
\ee
Here $SM$ denotes the spinor bundle over $M$ and $SB$ the restriction of $SM$ on $B \subset M$.
\end{prop}

Proposition $\ref{prop-refined-weighted-elliptic-estimate}$ is an immediate consequence of Proposition $\ref{prop-weight-elliptic-estimate}$ (with $p=2$) and the following Lemma:
\begin{lem}\label{lem-refined-weighted-elliptic-estimate}
If $(\delta, \beta) \in \R^2$ is not critical as in Definition $\ref{defn-critical-index}$, there exists a constant $C$ and a compact set $B \subset M \setminus \{o\}$ such that for any $\varphi \in L^2_{\delta, \beta}(SM)$ with $D \varphi \in L^{2}_{\delta, \beta}(SM)$,
\be
\|\varphi\|_{L^{2}_{\delta, \beta}(SM)} \leq C \left( \|D \varphi \|_{L^{2}_{\delta-1, \beta-1}(SM)} + \|\varphi\|_{L^{2}(SB)} \right).
\ee
Here $SM$ denotes the spinor bundle over $M$ and $SB$ the restriction of $SM$ on $B \subset M$.\end{lem}

\begin{proof}
    We only need to deal with the model metric $\Bar{g}$ on the conical neighborhood of the singular point $o$ and $g_{\mathbb{R}^{n}}$ near the infinity. And we will only provide details for the estimate on the conical neighborhood. The estimate near the infinity can be done similarly. 
    
    In the rest of the proof, we work on the model cone $((0, 1)\times N, \g=dr^2 + r^2 g^N)$.  Let  $\{\varphi_j\}^{+\infty}_{j = - \infty}$ be an orthonormal basis of $L^2(SN)$ consisting of eigenspinors of $D^{g^{N}}$ with eigenvalues $\lambda_{j}$. Then by the anti-commutativity in $(\ref{eqn-partial-r-D-anti-commute})$, $\{\partial_r \cdot \varphi_j\}^{+\infty}_{j = - \infty}$ is another orthonormal basis of $L^2(SN)$ consisting of eigenspinor of $D^{g^N}$ with eigenvalues $-\lambda_j$.

    For a spinor $\varphi$ on the cone $((0, 1) \times N, \overline{g}) $, let $\psi = \overline{D} \varphi$. We write
    \begin{equation} \label{def-spectral-decomposition}
    \varphi = \sum^{+ \infty}_{j = -\infty} u_j(r) \varphi_j, \ \ \text{and} \ \ 
    \psi = \sum^{+\infty}_{j=-\infty} v_j(r) \partial_r \cdot \varphi_j.
    \end{equation}
    Then by $(\ref{eqn-Dirac-op-relation})$, the equation $\overline{D} \varphi = \psi$ is equivalent to 
    \begin{equation}\label{eqn-spectral-decomposition}
    \frac{d}{dr}u_j(r) + \frac{n-1}{2r}u_j(r) + \frac{\lambda_j}{r}u_j(r) = v_j(r), \ \ j \in \mathbb{Z}.
    \end{equation}
    Let $r = e^s, \tu_j(s) = u_j(e^s), \tv_j(s) = e^{s}v_j(e^s)$, then
    \begin{equation}
    \tu^\prime_j(s) = \left(\frac{du_j}{dr}\right)(e^s)e^s.
    \end{equation} 
    We will use `` $\prime$ " to denote the derivative with respect to the variable $s$. Note that $s \in (-\infty, 0)$. By doing this change of variable, the equation (\ref{eqn-spectral-decomposition}) becomes
    \begin{equation} \label{eqn-spectral-decomposition1}
    \tu^\prime_j(s) + \left( \frac{n-1}{2} + \lambda_j\right) \tu_j(s) = \tv_j(s).
    \end{equation}



Take a cut-off function $\chi$ defined on $(-\infty, 0)$, which vanishes on $(-1, 0)$ and is equal to $1$ on $(-\infty, -2)$. Integration by parts provides
  \begin{eqnarray*}
    & & \left(\frac{n-1}{2}+\lambda_{j}\right)\int_{-\infty}^{0}\left(\chi \tu_{j}\right)^2 e^{-2\delta s}ds \\
     & = & \int_{-\infty}^{0}\left(\chi^2 \tu_{j} \tv_{j}\right) e^{- 2\delta s}ds -\int_{-\infty}^{0}\left(\chi^2 \tu_{j}\right) \tu'_{j}e^{(-2\delta)s}ds \\     
     &=& \int_{-\infty}^{0}\left(\chi^2 \tu_{j} \tv_{j}\right)e^{-2\delta s} ds -\frac{1}{2}\int_{-\infty}^{0}\chi^2 \left(\tu^2_{j}\right)'e^{(-2\delta)s}ds\\
     &=& \int_{-\infty}^{0}\left(\chi^2 \tu_{j}\tv_{j}\right)e^{-2\delta s} ds +\frac{1}{2}\int_{-\infty}^{0}\left(\chi^2\right)' \tu^2_{j}e^{-2\delta s}ds - \delta \int_{-\infty}^{0}\left(\chi^2\right) \tu^2_{j}e^{-2\delta s}ds\\
     &=& \int_{-\infty}^{0}\left(\chi^2 \tu_{j}\tv_{j}\right) e^{-2\delta s} ds +\frac{1}{2}\int_{-\infty}^{0}\left(\chi^2\right)' \tu^2_{j} e^{-2\delta s} ds - \delta \int_{-\infty}^{0}\left(\chi^2\right)\tu^2_{j} e^{-2\delta s} ds.
  \end{eqnarray*}
Therefore, we have
  \begin{eqnarray*}
     \left(\frac{n-1}{2} + \lambda_j + \delta \right)\int_{-\infty}^{0}\left(\chi \tu_{j}\right)^2 e^{-2 \delta s} ds
     & = &\int_{-\infty}^{0}\left(\chi^2 \tu_{j} \tv_{j}\right) e^{-2\delta s}ds +\frac{1}{2}\int_{-\infty}^{0}\left(\chi^2\right)' \tu^2_{j} e^{-2\delta s} ds.
  \end{eqnarray*}
  Thus, if $\delta \neq - \frac{n-1}{2} -\lambda_j $ for all $j \in \mathbb{Z}$, we have
  \begin{eqnarray}
     \int_{-\infty}^{0}\left(\chi \tu_{j}\right)^2 e^{-2\delta s} ds
     & \leq &
     C \int_{-\infty}^{0}\left(\chi^2 \tv^2_{j}\right) e^{-2 \delta s} ds + C \int_{-2}^{-1}\tu^2_{j} e^{-2\delta s} ds. \label{ellest}
  \end{eqnarray}
  for some constant $C$ independent of $j$. Then changing the variable back to $r$, we obtain the estimate near the conical point. Similarly, we can obtain the estimate at infinity, and complete the proof.
    \end{proof}

\subsection{Fredholm property of the Dirac operator}\label{subsect-Fredholm}
We consider the unbounded operator 
\begin{eqnarray*}
D_{\delta, \beta}: {\rm Dom}(D_{\delta, \beta}) & \rightarrow & L^2_{\delta-1, \beta-1}(SM) \\
 \varphi & \mapsto & D \varphi,
\end{eqnarray*}
whose domain ${\rm Dom}(D_{\delta, \beta})$ is dense subset of $L^2_{\delta,\beta}(SM)$ consisting of spinors $\varphi$ such that $D\varphi \in L^2_{\delta, \beta}(SM)$ (in the sense of distributions).

\begin{lem}
The unbounded operator $D_{\delta, \beta}$ is closed.
\end{lem}
\begin{proof}
Let ${\varphi_i} $ be a sequence in $L^{2}_{\delta, \beta}$, $\varphi_0 \in L^{2}_{\delta, \beta}(SM)$, and $\psi_{0} \in L^{2}_{\delta-1, \beta-1}(SM)$ such that 
\be
\varphi_i \rightarrow \varphi_0, \ \ \text{in} \ \ L^{2}_{\delta, \beta}(SM), \ \ \text{as} \ \ i \rightarrow \infty,
\ee
and 
\be
D\varphi_i \rightarrow \psi_0, \ \ \text{in} \ \ L^2_{\delta-1, \beta -1}(SM), \ \ \text{as} \ \ i \rightarrow \infty.
\ee
Then by the weighted elliptic estimate in Proposition \ref{prop-weight-elliptic-estimate}, we have that $\{\varphi_i\}$ is a Cauchy sequence in $W^{1, 2}_{\delta, \beta}(SM) \subset L^{2}_{\delta, \beta}$. As a result, $\varphi_0 \in W^{1, 2}_{\delta, \beta}(SM)$ and $D \varphi_0 = \psi_0$, since
\begin{eqnarray*}
& & \|D \varphi_0 - \psi_0\|_{L^2_{\delta-1, \beta-1}(SM)} \\
& \leq & \|D \varphi_0 - D \varphi_i\|_{L^2_{\delta-1, \beta-1}(SM)} + \|D \varphi_i - \psi_0\|_{L^2_{\delta-1, \beta-1}(SM)} \\
& \leq & C \|\varphi_0 - \varphi_i\|_{W^{1, 2}_{\delta, \beta}(SM)} + \|D \varphi_i - \psi_0\|_{L^2_{\delta-1, \beta-1}(SM)} \\
& \rightarrow & 0, \ \ \text{as} \ \ i \rightarrow \infty.
\end{eqnarray*}
\end{proof}

The usual $L^2$ pairing $(\cdot, \cdot)_{L^2(SM)}$ identifies the topological dual space of $L^2_{\delta, \beta}$ with $L^{2}_{-\delta-n, - \beta -n}$. For this identification, the adjoint $\left( D_{\delta, \beta} \right)^*$ of $D_{\delta, \beta}$ is
\begin{eqnarray*}
\left( D_{\delta, \beta} \right)^* : {\rm Dom}\left((D_{\delta, \beta})^{*}\right) & \rightarrow & L^{2}_{-\delta-n, -\beta - n} \\
  \psi & \mapsto & D \psi,
\end{eqnarray*}
where the domain ${\rm Dom}\left( \left( D_{\delta, \beta} \right)^* \right)$ is the dense subset of $L^{2}_{-\delta+1-n, -\beta+1-n}$ consisting of spinors $\psi$ such that $D \psi \in L^2_{-\delta-n, \beta-n}$ (in the distributional sense).

\begin{prop}\label{prop-Fredholm}
If $(\delta, \beta) \in \R^2$ is not critical as in Definition $\ref{defn-critical-index}$, then the operator $D_{\delta, \beta}$ is Fredholm, namely, 
\begin{enumerate}[$(1)$]
\item ${\rm Ran}(D_{\delta, \beta})$ is closed,
\item ${\rm dim}\left( {\rm Ker}(D_{\delta, \beta}) \right) < +\infty$,
\item ${\rm dim}\left( {\rm Ker}((D_{\delta, \beta})^*) \right) < +\infty$.
\end{enumerate}
\end{prop}
\begin{proof}
\begin{enumerate}[(1)] \item  The closeness of ${\rm Ran}(D_{\delta, \beta})$ follows from (2) and (3), namely ${\rm dim}\left( {\rm Ker}(D_{\delta, \beta}) \right) < +\infty$ and ${\rm dim}\left( {\rm Ker}((D_{\delta, \beta})^*) \right) < +\infty$ \cite[ p. 156]{AbramovichAliprantis}.
    \item 
    Clearly, $ {\rm Ker}(D_{\delta, \beta}) $ consists of spinors $\varphi \in L^2_{\delta, \beta}(M)$ such that $D \varphi =0$. Thus, by the standard elliptic regularity theory, ${\rm Ker}(D_{\delta, \beta}) \subset \left( C^\infty (SM) \cap L^2_{\delta, \beta}(SM) \right)$. Moreover, by the weighted elliptic estimate in Proposition \ref{prop-weight-elliptic-estimate}, we have
\be
{\rm Ker}(D_{\delta, \beta}) \subset W^{1, 2}_{\delta, \beta}(SM).
\ee
Thus ${\rm Ker}(D_{\delta, \beta})$ is the same as the kernel ${\rm Ker}\left( \widetilde{D_{\delta, \beta}} \right)$ of the bounded operator 
\begin{eqnarray}
\widetilde{D_{\delta, \beta}}: W^{1, 2}_{\delta, \beta}(SM) & \rightarrow & L^2_{\delta, \beta}(SM) \\
 \varphi & \mapsto & D \varphi.
\end{eqnarray}
As a result, ${\rm Ker}(D_{\delta, \beta}) = {\rm Ker}\left( \widetilde{D_{\delta, \beta}} \right)$ is closed in $W^{1, 2}_{\delta, \beta}(SM)$. Thus, it suffices to show that the unit sphere in ${\rm Ker}(D_{\delta, \beta})$ as a subspace of $W^{1, 2}_{\delta, \beta}(SM)$ is compact. This then follows from the refined weighted elliptic estimate in Proposition \ref{prop-refined-weighted-elliptic-estimate} and the compactness of the embedding $W^{1, 2}_{\delta, \beta}(SM) \subset L^2(SB)$.

\item Finally, note that ${\rm Ker}((D_{\delta, \beta})^*) \subset {\rm Ker}(D_{\delta^\prime, \beta^\prime}) $ for some small non-critical $\delta^\prime$ and large non-critical $\beta^\prime$. Then ${\rm dim}\left( {\rm Ker}((D_{\delta, \beta})^*) \right) < +\infty$ follows from (2). 
\end{enumerate}
\end{proof}


\subsection{Solving the equation $D \varphi=\psi$.}

Let $(M^n, g, o)$ be a AF manifold with a single conical singularity at $o$ as defined in Definition \ref{defn-AF-conic-mfld}. 

We first solve the equation $D \varphi = \psi $ near the conical point $o$. The equation can be solved near infinity similarly. We will use $B_r$ to denote the ball centered at the conical singularity $o$ with radius $r$, and $SB_r$ to denote the spinor bundle over it. In the notations for weighted Sobolev norms and spaces over $B_r$, the subscript $\beta$ will be neglected, and they will be written as $\|\cdot\|_{W^{1, 2}_{\delta}(SB_r)}$ and $W^{1, 2}_{\delta}(SB_r)$. 

For the Dirac operator $\overline{D}$ of the model cone metric $\overline{g} = dr^2 + r^2 g^N$, we have the following lemma.
\begin{lem}\label{lemma-G}
  Given $\delta$, which is noncritical at conical point, and a small number $r_{0} >0$, there is a bounded operator
  \[G:L^{2}_{\delta-1}(SB_{2r_{0}})\to W^{1,2}_{\delta}(SB_{2r_{0}})\]
   such that $\D\circ G=id$. 
\end{lem}
\begin{proof}
  We define $G$ as following. 
  
  We pick a spinor $\psi \in L^{2}_{\delta-1}(SB_{2r_0})$ .  Given $r<<2r_{0}$, we can use standard elliptic theory to solve the equation $\overline{D} \varphi_{r} = \psi$ in $L^{2}(S(B_{2r_{0}}\backslash B_{r}))$, with Dirichlet boundary condition. Adapting the proof of Lemma \ref{lem-refined-weighted-elliptic-estimate} (on $B_{2r_{0}}\backslash B_{r}$ and with $\chi=1$, more specifically using the equation immediately before \eqref{ellest} and using Cauchy-Schwarz), we obtain
  \begin{equation}\label{eqn-solve-u=f}
    \|\varphi_{r}\|_{L^2_{\delta}(S(B_{2r_{0}}\backslash B_{r}))}\leq c\|\psi\|_{L^2_{\delta-1}(S(B_{2r_{0}}\backslash B_{r}))},
  \end{equation}
  with $c$ independent of $r$. Elliptic regularity then bounds the $W^{1,2}$ norm of $\varphi_{r}$ over compact subsets in terms of $\|\psi\|_{L^2_{\delta-1}}$, so that we can use Rellich theorem and a diagonal argument to extract a sequence $\varphi_{r}$ converging to a spinor $\varphi$ in $L^2_{loc}$, with $\D \varphi = \psi$ and $\varphi=0$ on $\partial B_{2r_{0}}$. Taking a limit in (\ref{eqn-solve-u=f}), one gets
  \[\|\varphi\|_{L^2_{\delta}(SB_{2r_{0}})}\leq c\|\psi\|_{L^2_{\delta-1}(SB_{2r_{0}})}.\]
  From the Proposition \ref{prop-weight-elliptic-estimate} and standard elliptic arguments near $\partial B_{2r_{0}}$, we deduce an estimate on the derivatives:
  \begin{equation}\label{eqn-solve-estimate-u}
    \|\varphi\|_{W^{1,2}_{\delta}(SB_{2r_{0}})}\leq c\|\psi\|_{L^2_{\delta-1}(SB_{2r_{0}})}.
  \end{equation}

  Finally, we show that such spinor $\varphi$ is uniquely defined, i.e. independent of the choice of extracted sequence. The difference $\bar{\varphi}$ between two such spinors $\varphi$ solves the Dirac equation $\overline{D} \bar{\varphi}=0$ and vanishes on $\partial B_{2r_{0}}$. Then by doing the spectral decomposition as in \eqref{def-spectral-decomposition}, 
  the equation $\overline{D} \bar{\varphi} = 0$ reduces to ODEs \eqref{eqn-spectral-decomposition} with $v_j=0$. 
  Together with the vanishing condition, one quickly deduces that $\bar{\varphi}=0$. We can therefore set $G \psi := \varphi$.
\end{proof}


A perturbation argument extends this result to a more general setting.
\begin{prop}\label{prop-G-g}
  Given a noncritical $\delta$ and a small number $r_{0}$, we can define a bounded operator $G_{g}:L^{2}_{\delta-1}(SB_{2r_{0}})\mapsto W^{1,2}_{\delta}(SB_{2r_{0}})$ such that $D\circ G_{g}=id$.
\end{prop}
\begin{proof}
  By using the bounded operator $G$ obtained in Lemma \ref{lemma-G}, we can write $D=\D[id+G(D-\D)]$. By the asymptotic control assumption of the difference between $g$ and $\overline{g}$ in $(\romannumeral3)$ of Definition \ref{defn-conic-mfld}, we can estimate $\left\|(\D-D)\varphi\right\|_{L^{2}_{\delta}(SB_{2r_{0}})}$ by $\epsilon_{2}(r_{0})\left\|\varphi\right\|_{W^{1,2}_{\delta}(SB_{2r_{0}})}$ for $\epsilon_{2}(r_{0})\to 0$ as $r_{0}\to 0$. Since $G$ is bounded from $L^{2}_{\delta-1}(SB_{2r_{0}})$ to $W^{1,2}_{\delta}(SB_{2r_{0}})$, we deduce that $G(D-\D)$ defines a bounded operator on $W^{1,2}_{\delta}(SB_{2r_{0}})$, whose norm goes to $0$ as $r_{0}$ approaches $0$. So for a sufficiently small $r_0$, $id+G(D-\D)$ is an automorphism of $W^{1,2}_{\delta}(SB_{2r_{0}})$ and $G_{g}=[id+G(D-\D)]^{-1}G$ is a bounded operator from $L^{2}_{\delta-1}(SB_{2r_{0}})$ to $W^{1,2}_{\delta}(SB_{2r_{0}})$, with $DG_{g}=G_{g}D=id$.
\end{proof}


As an application of Proposition $\ref{prop-G-g}$, we obtain harmonic spinors near the conical point, with certain prescribed asymptotic. Let $E_j$ be the eigenspace of the Dirac operator $D^{g^N}$ on the cross section $N$ of the model cone, corresponding to eigenvalues $\lambda_j$. 

\begin{cor}\label{cor-harmonic-spinor-on-cone}
Given $j\in \mathbb{N}$ and $\phi\in E_{j}$, there are spinors $\mathcal{H}_{j,\phi}$ that are harmonic, that is $D \mathcal{H}_{j, \phi} = 0$, near the conical point and have the form \[\mathcal{H}_{j,\phi}=r^{\nu_{j}}\phi+ \psi\]
with $ \psi $ in $W^{1,2}_{\eta}$ for any $\eta< \nu_j + \alpha := -\lambda_{j} - \frac{n-1}{2} + \alpha $. Here $\alpha$ is the decay order of the metric $g$ near the conical point as in Definition $\ref{defn-conic-mfld}$.
\end{cor}
\begin{proof}
  For noncritical $\delta < -\lambda_j - \frac{n-1}{2}$, one can easily check that $\overline{D}(r^{\nu_j}\phi) = 0$ and $r^{\nu_j} \phi \in W^{1, 2}_\delta$. Because the operator $D - \overline{D}: W^{1, 2}_{\delta} \rightarrow L^2_{\delta - 1 + \alpha}$ is bounded, we further have $D(r^{\nu_j} \phi) \in L^2_{\delta-1+\alpha}$. Then Proposition \ref{prop-G-g} implies that there exists $\varphi \in W^{1, 2}_{\delta + \alpha}(SB_{2r_0})$ for some small $r_0>0$ such that $D \varphi  = D(r^{\nu_j}\phi)$. Now we set $\mathcal{H}_{j, \phi}:=\chi \left( r^{\nu_j}\phi + \varphi \right)$ for some smooth cut-off function $\chi$ which vanishes on $B^c_{2r_0}$ and is equal to 1 on $B_{r_0}$. This $\mathcal{H}_{j, \phi}$ satisfies requirements in the conclusion.
\end{proof}


Now we derive a decay jump property for solutions to $D \varphi = \psi$ near the conical point, following the arguments in Lemma 5 and Proposition 4 in \cite{Minerbe-CMP} that is decay jump property at infinity. This decay jump property is critical ingredient to extend the range of noncritical indices $(\delta, \beta)$, for which the map $D_{\delta, \beta}$ is surjective, from Proposition $\ref{prop-surjectivity}$ to Proposition $\ref{prop-surjectivity-1}$.

We first work on the the Dirac operator $\overline{D}$ of the model cone metric $\overline{g}$.

\begin{lem}\label{lem-decay-jump-for-Dbar}
  Suppose $\D \varphi = \psi $ with $\varphi$ in $L_{\delta}^2(SB_{2r_{0}})$ and $\psi$ in $L^{2}_{\delta'-1}(SB_{2r_{0}})$ for non-critical exponents $\delta<\delta'$ and a small number $r_{0}$. Then there is an element $\varphi'$ of $L^{2}_{\delta'}(SB_{2r_{0}})$ such that $\varphi - \varphi'$ is a linear combination of the following functions:
  \[r^{\nu_{j}}\phi_{j} \ \ \;\text{with}\; \phi_{j}\in E_{j}\; \text{and}\; \delta<  \nu_j = -\lambda_{j} - \frac{n-1}{2}<\delta'.\]
\end{lem}
\begin{proof}
  We will build $\varphi'$ step by step, starting from the solution $\tilde{\varphi}$ of $\D \tilde{\varphi}=\psi$ provided by Lemma \ref{lemma-G}. Note $\tilde{\varphi} \in W^{1, 2}_{\delta'}(SB_{2r_{0}})$. The difference $\omega=\varphi-\tilde{\varphi}$ satisfies $\overline{D} \omega = 0$. We do the eigenspinor decomposition for $\omega$ with respect to the Dirac operator on the cross section $(N, g^N)$, and use $w_j$ to denote the component in the eigenspace $E_j$. The equation $\D \omega_{j}=0$ then implies $\omega_{j}=r^{\nu_{j}}\phi_{j}$ with $\phi_{j}$ in $E_{j}$. Observing that $r^{\nu_{j}} \phi_j\in W^{1, 2}_{\eta}\Leftrightarrow \eta< \nu_j$, one sees that each term is either in $W^{1, 2}_{\delta'}$, so that we can add it to $\tilde{\varphi}$ and forget it, or satisfies the conditions in the statement.
\end{proof}

Then we generalize this to the Dirac operator $D$ of an asymptotically conical metric.

\begin{prop}\label{prop-asmptotic-order}
  Suppose $D \varphi = \psi$ with $\varphi$ in $L_{\delta}^2(SB_{2r_{0}})$ and $\psi$ in $L^{2}_{\delta'-1}(SB_{2r_{0}})$ for non-critical exponents $\delta<\delta'$ . Then, up to making $r_0$ smaller, 
  there is an element $\varphi'$ of $L^{2}_{\delta'}(SB_{2r_{0}})$ such that $\varphi - \varphi'$ is a linear combination of the following spinors:
  
  \[\mathcal{H}_{j,\phi_{j}} \;\text{with}\;  \phi_{j} \in E_{j} \;\text{and}\;  \delta<-\lambda_{j}+\frac{1-n}{2}<\delta',\]
  where $\mathcal{H}_{j, \phi_j}$ is a spinor, harmonic near the conical point, obtained in Corollary $\ref{cor-harmonic-spinor-on-cone}$.
\end{prop}
\begin{proof}
    By using Proposition \ref{prop-refined-weighted-elliptic-estimate}, from the equation $D\varphi = \psi$, with $\varphi \in L^2_{\delta}$ and $\psi \in L^2_{\delta'-1}$, we have $\varphi \in W^{1,2}_{\delta}$. By the fact that the operator $D - \overline{D}: W^{1, 2}_\delta(SB_{2r_0}) \rightarrow L^2_{\delta - 1 + \alpha}$ is bounded, we further have
    \[\overline{D}\varphi = D \varphi + (\overline{D} - D) \varphi \in L^2_{\delta'-1}+L^2_{\delta+\alpha-1}.\]
    So if we pick any noncritical $\eta\leq \min\left\{\delta',\delta+\alpha\right\}$, we have $\overline{D}\varphi \in L^{2}_{\eta-1}$. Lemma \ref{lem-decay-jump-for-Dbar} then implies that $\varphi$ admits a decomposition
    \[\varphi =\varphi_{1}+\sum_{j}r^{\nu_{j}}\phi_{j},\]
    where $\varphi_{1}\in L^{2}_{\eta}$, and the second term is a sum of finitely many terms with $j$ such that $\delta< \nu_j = -\lambda_{j} - \frac{n-1}{2}<\eta$. By applying Corollary $\ref{cor-harmonic-spinor-on-cone}$, we can further write
    \[\varphi=\varphi_{2}+\sum_{j}\mathcal{H}_{j,\phi_{j}},\]
    where $\varphi_{2}\in L^{2}_{\eta}$, $\delta<-\lambda_{j}+\frac{1-n}{2}<\eta$.
    
    If $\delta+\alpha\geq \delta'$, we are done. 
    
    If not, $D\varphi_{2}=\psi$ near the conical part and $\varphi_{2}\in L^2_{\eta}$. So we can repeat the argument with $\varphi_{2}$ in the role of $\varphi$ and $\eta$ in the role of $\delta$. In a finite number of steps, we are in the first case.
\end{proof}


Finally, we derive the surjectivity of $D_{\delta, \beta}$ for certain noncritical indices $(\delta, \beta)$, which enables us to solve the equation $D \varphi = \psi$.

\begin{prop}\label{prop-surjectivity-1}
Let $(M^n, g, o)$, $n \geq 3$, be an AF spin manifold with a single conical singularity at $o$. If the scalar curvature ${\rm Scal}_g \geq 0$, then for any noncritical $(\delta, \beta)$ satisfying $\delta \leq -\frac{n}{2}$ and $\beta \geq -\frac{n}{2}$,
the map
\be
D_{\delta, \beta}: {\rm Dom}\left( D_{\delta, \beta} \right) \rightarrow L^2_{\delta -1, \beta -1}(SM)
\ee
is surjective.

\end{prop}
\begin{proof}
We will first show that for $\delta\leq - \frac{n}{2}$ and $\beta \geq - \frac{n}{2}$, ${\rm Ker}\left( \left( D_{\delta, \beta} \right)^* \right) = \{0\}$. If we further assume that $(\delta, \beta)$ is noncritical, then by using the Fredholm property obtained in Proposition $\ref{prop-Fredholm}$, the surjectivity of the map $D_{\delta, \beta}$ follows. 

First, we show that ${\rm Ker}\left( \left( D_{\delta, \beta} \right)^* \right) = \{0\}$ 
for any $\delta \leq -\frac{n}{2}$ and $\beta \geq -\frac{n}{2}$. For such $\delta$ and $\beta$, we have
\be
-\delta + 1 -n \geq \frac{2-n}{2}, \ \ \text{and} \ \ -\beta + 1 - n \leq \frac{2-n}{2}.
\ee
Thus
\be
{\rm Dom}\left( \left( D_{\delta, \beta} \right)^* \right) \subset L^{2}_{-\delta + 1 -n, \beta + 1 - n}(SM) \subset L^2_{\frac{2-n}{2}, \frac{2-n}{2}}(SM).
\ee
For every small number $r_0$ and large number $R_0$, we can choose a smooth cut-off function 
\be
0 \leq \chi_{r_0, R_0} \leq 1
\ee
such that
\be
\chi_{r_0, R_0} 
= \begin{cases} 
    0, & \text{on} \ \ B_{r_0}(o) \cup B^c_{2R_0}, \\
    1, & \text{on} \ \ ( B_{2r_0}(o))^c \cap B_{R_{0}},
   \end{cases}
\ee
and
\be
|d \chi_{r_0, R_0}| 
\leq  \begin{cases}
        \frac{10}{r_0}, & \text{on} \ \ A_{r_0}, \\  
        \frac{10}{R_0}, & \text{on} \ \ A_{R_0},
       \end{cases}
\ee
where $A_{r_0}:= B_{2r_0}(o) \setminus B_{r_{0}}(o)$ and $A_{R_0}:= B_{2R_0}(o) \setminus B_{R_0}(o)$.
Then for $\varphi \in {\rm Ker}\left( \left( D_{\delta, \beta} \right)^* \right) \subset L^{2}_{\frac{2-n}{2}, \frac{2-n}{2}}(SM)$, by doing integration by parts and using Lichnerowicz-Bochner formula, we have
\begin{eqnarray*}
\int_{M} |  \nabla (\chi_{r_0, R_0} \varphi)|^2 
& = &  \int_M |d\chi_{r_0, R_0}|^2 |\varphi|^2 + \int_{M} \left( \chi_{r_0, R_0} \right)^2 \langle \varphi, D^2 \varphi \rangle - \frac{1}{4} \int_M \left(\chi_{r_0, R_0}\right)^2 {\rm Scal}_g |\varphi|^2 \\
& \leq & \int_M |d\chi_{r_0, R_0}|^2 |\varphi|^2 + \int_{M} \left( \chi_{r_0, R_0} \right)^2 \langle \varphi, D^2 \varphi \rangle
\end{eqnarray*}
implies
\begin{eqnarray}
\int_{(B_{2r_0}(o))^c \cap B_{R_0}} |\nabla \varphi|^2 
& \leq &  C \left( \int_{A_{r_0}} |\varphi|^2 r^{-2} + \int_{A_{R_0}} |\varphi|^2 \rho^{-2} \right) \\
& \rightarrow & 0, \ \ \text{as} \ \ r_0 \rightarrow 0, \ \ \text{and} \ \ R_{0} \rightarrow +\infty,
\end{eqnarray}
since 
\be
\| \varphi \|^2_{L^2_{\frac{2-n}{2}, \frac{2-n}{2}}(M)} =  \int_{M}  \left( r^{-2} | \varphi |^2 \chi_1 + \rho^{-2} |\varphi |^2 \chi_2 + | \varphi |^2 \chi_3 \right) d\vol_{g} < \infty.
\ee
Thus $\nabla \varphi =0$, and so $\varphi$ is constant. Then $\varphi \in L^2_{\frac{2-n}{2}, \frac{2-n}{2}}(SM)$ implies $\varphi = 0$, since the improper integral $\int^{+\infty}_{R_0} \rho^{-(2-n)-n} \cdot \rho^{n-1}d\rho = \int^{+\infty}_{R_0}\rho^{n-3}d\rho$ is divergent for $n \geq 3$. This shows:  ${\rm Ker}\left( \left( D_{\delta, \beta} \right)^* \right) = \{0\}$.

Now we show the surjectivity of $D_{\delta, \beta}$ for noncritical $(\delta, \beta)$ satisfying $\delta \leq -\frac{n}{2}$ and $\beta \geq -\frac{n}{2}$, by using ${\rm Ker}\left( \left( D_{\delta, \beta} \right)^* \right) = \{0\}$ and Proposition \ref{prop-Fredholm}. For any fixed $\psi \in C^\infty_0 (SM)$, if 
\be
\left(\psi, D_{\delta, \beta} \varphi \right)_{L^2(SM)} = 0,  \ \ \forall \varphi \in {\rm Dom}\left( D_{\delta, \beta}  \right),
\ee
then
\be
\left( \left( D_{\delta, \beta} \right)^* \psi, \varphi \right)_{L^2(SM)} = 0, \ \ \forall \varphi \in {\rm Dom}\left( D_{\delta, \beta}  \right).
\ee
By the density of $C^\infty_0(SM) \subset {\rm Dom}\left( D_{\delta, \beta} \right)$ in $L^2(SM)$, this implies that $\left( D_{\delta, \beta} \right)^* \psi = 0$, and so $\psi = 0$. Therefore, $\left( {\rm Ran}\left( D_{\delta, \beta} \right) \right)^\perp \cap  C^\infty_0(SM)$ = \{0\}. As a result, the density of $C^\infty_0(SM)$ and closeness of ${\rm Ran}\left( D_{\delta, \beta}\right)$ in $L^2_{\delta-1, \beta-1}(SM)$ (Proposition \ref{prop-Fredholm})  implies that ${\rm Ran}\left( D_{\delta, \beta} \right) = L^2_{\delta-1, \beta-1}(SM)$, i.e. $D_{\delta, \beta}$ is surjective.
\end{proof}

\begin{defn}\label{defn-delta_0}
Let $\lambda_{-1}$ denote the largest negative eigenvalue of the Dirac operator $D^{g^N}$. Then we define
\begin{equation}
\delta_0 := \frac{1-n}{2} - \lambda_{-1}.
\end{equation}
\end{defn}
Then the estimate of eigenvalues in Lemma $\ref{lem-cross-section-eigenvalue-lower-bound}$ implies the nonnegativity of $\delta_0$.
\begin{lem}\label{lem-delta_0}
Let $(M^n, g, o)$, $n \geq 3$, be an AF spin manifold with a single conical singularity at $o$ whose scalar curvature $\Scal_g \geq 0$. Then $\delta_0$ defined in Definition \ref{defn-delta_0} is nonnegative. Moreover, $\delta_0 = 0$ if and only if the cross section $(N, g^N)$ admits a real Killing spinor.
\end{lem}

Finally, we improve the estimate, in Proposition $\ref{prop-surjectivity-1}$, of the range of indices $\delta$ and $\beta$, for which the map $D_{\delta, \beta}$ is surjective as following.
\begin{prop}\label{prop-surjectivity}
If the scalar curvature ${\rm Scal}_g \geq 0$, then 
\be
D_{\delta, \beta}: {\rm Dom}\left( D_{\delta, \beta} \right) \rightarrow L^{2}_{\delta-1,\beta-1}(SM)
\ee
is surjective for noncritical $\delta\leq \delta_0$ and $\beta\geq 1-n$.
\end{prop}
\begin{proof}
By Proposition \ref{prop-surjectivity-1}, it suffices to show that $D_{\delta, \beta}$ is surjective for $-\frac{n}{2} < \delta < \delta_0$ and $1-n < \beta < -\frac{n}{2}$. 

For arbitrary noncritical $\delta^\prime \in \left( -\frac{n}{2}, \delta_0 \right)$ and $\beta^\prime \in \left( 1-n, -\frac{n}{2}\right)$, we take an arbitrary spinor $ \psi \in L^2_{\delta^\prime -1, \beta^\prime -1}(SM) \subset L^2_{-\frac{n}{2}-1, -\frac{n}{2}-1}(SM)$. Proposition \ref{prop-surjectivity-1} implies that there exists $\varphi \in L^2_{-\frac{n}{2}, -\frac{n}{2}}(SM)$ such that $D \varphi = \psi$. Proposition \ref{prop-asmptotic-order} then implies that $ \varphi \in L^2_{\delta^\prime, \beta^\prime}$, since there is no critical index at  conical point in $\left( -\frac{n}{2}, \delta^\prime \right)$ and no critical index at infinity in $\left( \beta^\prime, -\frac{n}{2} \right)$. Therefore, $D_{\delta^\prime, \beta^\prime}$ is surjective, and this completes the proof.
\end{proof}

\section{Solving Dirac equation}\label{sect-sovlving-Dirac-equation} \label{sec-solv}
In this section, we solve the Dirac equation $D \varphi = 0$ to obtain a harmonic spinor, which has certain asymptotic control near the conical point and asymptotic to a constant spinor near the infinity. This harmonic spinor plays a critical role in the calculation of the mass and the proof of the positive mass theorem in \S$\ref{sect-PMT}$.

Let $\psi_0$ be a constant spinor on $\R^n$. As in \S\ref{subsect-weighted-Sobolev-space}, $\chi_2$ denotes a truncation function supported in $M_\infty$. By identifying $M_\infty$ with $\R^n \setminus B_{R}(0)$, then $\chi_2 \psi_0$ is a spinor on $M^n$ and vanishes in a neighborhood of the conical singularity $o$, and 
\be
D (\chi_2 \psi_0) = O(\rho^{-\tau -1}), \ \ \text{as} \ \ \rho \rightarrow \infty.
\ee
Clearly,  $D (\chi_2 \psi_0)$ vanishes in a neighborhood of $o$.
Thus
\be
D(\chi_2 \psi_0) \in L^2_{\delta-1, \beta-1}, \ \ \forall \delta \in \R \ \ \text{and} \ \ \beta > -\tau.
\ee
Then because $\tau > \frac{n-2}{2}$, in particular, 
\begin{equation}
    D(\chi_2 \psi_0) \in L^2_{\delta-1, \beta-1}(SM), \quad \forall \delta < \delta_0, \ \  -\tau < \beta < -\frac{n-2}{2}.
\end{equation}

Now Proposition $\ref{prop-surjectivity}$ implies that there exists $\xi \in L^{2}_{\delta, \beta}(SM)$, for noncritical indices $(\delta, \beta)$ with $\delta < \delta_0$ and $\max\{-\tau,1-n\} < \beta < -\frac{n-2}{2}$, such that
\be\label{eqn-D}
D \xi = D (\chi_2 \psi_0).
\ee
We set 
\begin{equation}\label{eqn-harmonic-spinor}
\varphi : = \chi_2 \psi_0 - \xi.
\end{equation}
Then equation $(\ref{eqn-D})$ says $D \varphi =0$. 

We obtain the following asymptotic estimate for $\xi$ at infinity, which will be curcial in the proof of the nonnegativity of the mass in Theorem \ref{thm-PMT}, in particular, in the derivation of Lemma \ref{lem-boundary-limit-mass}.
\begin{lem}\label{lem-asymptotic-xi}
The solution $\xi$ to the equation $(\ref{eqn-D})$ satisfies: 
at infinity, 
\begin{equation}
 |\nabla^i \xi| = o(\rho^{\beta - i}), \ \ \text{as} \ \ \rho\rightarrow+\infty,
\end{equation}
for $\max\{-\tau, 1-n\} < \beta < -\frac{n-2}{2}$ and $i=0, 1$. 
\end{lem}
\begin{proof}
Because $D(\chi_2 \psi_0)$ is smooth on $M\setminus\{o\}$, by local elliptic regularity theory, we know $\xi$ is smooth on $M \setminus \{o\}$. By applying $D$ on the both sides of the equation $(\ref{eqn-D})$, we see that $\xi$ is a solution to the second order elliptic equation:
\begin{equation}
D^2 \xi = D^2(\chi_2 \psi_0) \in L^{2}_{\delta-1, \beta-2},
\ \ \forall \delta<\delta_0, \ \ -\tau < \beta < -\frac{n-2}{2}.
\end{equation}
Note that, by the asymptotic control for the metric $g$ near infinity in Definition $\ref{defn-AF-conic-mfld}$, we have
\begin{equation}
D^2(\chi_2 \psi_0) = O(\rho^{-\tau -2}), \ \ \nabla \left( D^2(\chi_2 \psi_0) \right) = O(\rho^{-\tau - 3}), \ \ \text{as} \ \ \rho \rightarrow \infty.
\end{equation}
As a result, the norms of both $D^2(\chi_2 \psi_0)$ and $\nabla \left( D^2(\chi_2 \psi_0) \right)$ are bounded near infinity.

Then with the help of the weighted Sobolev inequality (see e.g. Theorem 1.2 in \cite{Bartnik-CPAM}), the asymptotic estimate of $\xi$ at infinity can be obtained by employing Nash-Moser iteration argument (e.g. as in Section 8 in \cite{DW-2022}). We will omit the details here as it is pretty standard but tedious.  
\end{proof}

In a sufficiently small neighborhood of the conical point, the right hand side of the equation $(\ref{eqn-D})$ is equal to zero. This enables us to derive the following partial asymptotic expansion for $\xi$ near the conical point, which will be crucial in the proof of both nonnegativity part and rigidity part of positive mass theorem in Theorem \ref{thm-PMT}.
\begin{lem}\label{lem-asymptotic-xi1}
The solution $\xi$ to the equation $(\ref{eqn-D})$ satisfies
\begin{equation}\label{eqn-xi-precisely}
\xi = r^{\delta_0}\phi + \xi^\prime, \ \ \text{as} \ \ r\rightarrow0, 
\end{equation}
where $\phi$ is an eigenspinor of the Dirac operator $D^{g^N}$ on the cross section with eigenvalue $\lambda_{-1}$, and \begin{equation}\label{eqn-asymptotic-xi1}
|\nabla^i \xi^\prime| = o(r^{\delta^\prime -i}), \ \ \text{as} \ \ r \to 0, \ \ \text{for some}  \ \ \delta^\prime > \delta_0 \ \ \text{and} \ \ i=0, 1.
\end{equation}
\end{lem}
\begin{proof}
We will first use the decay jump property in Proposition \ref{prop-asmptotic-order} to obtain a decomposition of $\xi$ as in $(\ref{eqn-xi-precisely})$, and then use a weighted elliptic bootstrapping argument to derive asymptotic control in $(\ref{eqn-asymptotic-xi1})$. In this proof, we only consider the behavior of $\xi$ near the conical point $o$. So we restrict the problem to a small neighborhood of $o$, and in the notations of weighted spaces, we omit the weight index $\beta$ that concerns about the behavior at infinity.

    Choose noncritical $\delta$ and $\delta^\prime$ such that $\delta < \delta_0 < \delta^\prime < \delta + \alpha$ and $\delta_0$ is the only critical index in $(\delta, \delta^\prime)$ at conical point. Then $D\xi = D(\chi_2 \psi_0) \in L^2_{\delta^\prime -1}(SB_{r_0})$, since $D(\chi_2 \psi_0)$ vanishes near the conical singularity. Moreover, $\xi \in L^{2}_{\delta}(SB_{r_0})$. Therefore, Proposition $\ref{prop-asmptotic-order}$ and Corollary $\ref{cor-harmonic-spinor-on-cone}$ imply that
\begin{equation}\label{eqn-decomposition-xi}
\xi = r^{\delta_0}\phi + \xi^\prime, \ \ \text{as} \ \ r\rightarrow0, 
\end{equation}
where $\phi$ is a eigenspinor of the Dirac operator $D^{g^N}$ on the cross section with eigenvalue $\lambda_{-1}$, and 
\begin{equation}\label{eqn-xi-prime-regularity}
\xi^\prime \in L^{2}_{\delta^\prime}(SB_{r_0}).
\end{equation}

Because $D \xi = D(\chi_2 \psi_0)$ and $\chi_2 \psi_0 = 0$ on a small neighborhood of the singular point $o$, we have $D\xi = 0 $ on a neighborhood of $o$. Then by $(\ref{eqn-decomposition-xi})$, on a small neighborhood of $o$, we have
\begin{equation}\label{eqn-xi-prime}
D\xi^\prime = - D (r^{\delta_0} \phi) = (\overline{D} - D)(r^{\delta_0} \phi),
\end{equation}
since $\overline{D} (r^{\delta_0} \phi) = 0$. Moreover, by the asymptotic control for the metric $g$ near the conical singularity as in the item ${\rm (iii)}$ in Definition \ref{defn-conic-mfld}, we have 
\begin{equation}\label{eqn-Dirac-operator-difference}
(\overline{D} - D) (r^{\delta_0} \phi) \in W^{1, p}_{\delta - 1}, \ \ \forall \delta < \delta_0 + \alpha \ \ \text{and} \ \ \forall p >1.
\end{equation}
In particular,
\begin{equation}\label{eqn-initial-regularity}
 (\overline{D} - D)(r^{\delta_0} \phi) \in W^{1, p}_{\delta^\prime - 1}(SB_{r_{0}}), \ \ \forall p>1.
\end{equation}

Then from the equation $(\ref{eqn-xi-prime})$ and the regularity estimates $\ref{eqn-xi-prime-regularity}$ and $(\ref{eqn-initial-regularity})$, by using weighted elliptic estimate in Proposition \ref{prop-weight-elliptic-estimate} with $k=2, p=2$, we obtain: $\xi^\prime \in W^{2, 2}_{\delta^\prime}(SB_{r_0})$. Then weighted Sobolev embedding in Proposition 3.4 in \cite{DW-MRL-2020} implies $\xi^\prime \in L^{\frac{2n}{n-2}}_{\delta^\prime}(SB_{r_0})$. Now combining this with $(\ref{eqn-Dirac-operator-difference})$ (for $p=\frac{2n}{n-2}$), again by using the weighted elliptic estimate in Proposition $\ref{prop-weight-elliptic-estimate}$ with $k=2, p = \frac{2n}{n-2}$, we obtain $\xi^\prime \in W^{2, \frac{2n}{n-2}}_{\delta^\prime}(SB_{r_0})$. If $\frac{2n}{n-2} > n$, then weighted elliptic embedding in Proposition 3.4 in \cite{DW-MRL-2020} implies the asymptotic control in $(\ref{eqn-asymptotic-xi1})$. Otherwise, we can repeat the above elliptic bootstrapping process finitely many times to obtain $W^{2, p}_{\delta^\prime}$ with $p>n$. Then the asymptotic control in $(\ref{eqn-asymptotic-xi1})$ follows by Proposition 3.4 in \cite{DW-MRL-2020}.
\end{proof}


\section{Proof of the positive mass with conical singularity}\label{sect-PMT} \label{sec-comp}
In this section, we prove a positive mass theorem for AF spin manifold with isolated conical singularity of dimension $n \geq 3$.

First of all, we recall Lemma 4.1 in \cite{Dai-PMT} (see also Lemma 11 in \cite{Minerbe-CMP}) as:
\begin{lem}\label{lem-boundary-limit-mass}
For the harmonic spinor $\varphi$ obtained in $(\ref{eqn-harmonic-spinor})$, we have
    \begin{eqnarray}\label{eqn-orderinlarge}
    \lim_{R\to\infty}\int_{S_{R}}\langle \nabla_{\nu_{R}}\varphi,\varphi\rangle d A=\frac{1}{4}\omega_{n}m(g)    
    \end{eqnarray}
    where $d A$ is the area element and $\nu_{R}$ is the outer normal vector for $S_{R}$.

\end{lem}

Then we prove the following positive mass theorem.
\begin{thm}\label{thm-PMT}
Let $(M^n, g)$ be a $n$-dimensional AF spin manifold with finitely many isolated conical singularity and $n\geq 3$. If the scalar curvature is nonnegative on the smooth part, then the mass $m(g)$ is nonnegative. Furthermore, the mass $m(g)=0$ if and only if $(M^{n},g)\simeq (\mathbb{R}^{n},g_{\mathbb{R}^{n}})$, provided additionally that the cross sections of the model cones are simply connected when $n = 4k$ for $k \in \N$.
\end{thm}
\begin{proof} For simplicity we assume that $(M, g)$ has only a single conical singularity at $o$; the finitely many case can be treated with only notational changes.
    For $R$ large and $\epsilon$ small, take a compact set $\Omega$ such that $\partial \Omega=S_{R}$ and a ball $B_{\epsilon}(o)$. For the harmonic spinor $\varphi$ obtained in (\ref{eqn-harmonic-spinor}), by  Lichnerowicz-Bochner formula $D^2=\nabla^{\ast}\nabla+\frac{{\rm Scal}_{g}}{4}$ and integration by part, we have
    \begin{eqnarray*}
        \int_{\Omega\backslash B_{\epsilon}(o)}\left[|\nabla\varphi|^2+\frac{{\rm Scal}_{g}}{4}\varphi^2\right]d\operatorname{vol}_{g}&=&\int_{S_{R}}\langle \nabla_{\nu_{R}}\varphi,\varphi\rangle d A-\int_{\partial B_{\epsilon}(o)}\langle \nabla_{\nu_{\epsilon}}\varphi,\varphi\rangle d A,
    \end{eqnarray*}
    where $d A$ is the area element, $\nu_{R}$ is the outer normal vector for $S_{R}$ and $\nu_{\epsilon}$ is the outer normal vector for $\partial B_{\epsilon}(o)$.
    By Lemma \ref{lem-boundary-limit-mass}, the first boundary term on the right hand side tends to $\frac{1}{4} \omega_n m(g)$ as $R \rightarrow +\infty$.
    Moreover, for sufficiently small $\epsilon > 0$, $\varphi =- \xi$ by the definition of $\varphi $ in $(\ref{eqn-harmonic-spinor})$, so the asymptotic estimate for $\xi$ in Lemma $\ref{lem-asymptotic-xi}$ implies
    \begin{equation}
    \int_{\partial B_{\epsilon}(o)}\langle \nabla_{\nu_\epsilon} \varphi, \varphi \rangle \sim \epsilon^{2 \delta - 1} \cdot \epsilon^{n-1} = \epsilon^{2\delta + n - 2} \to 0, \ \ \text{as} \ \ \epsilon \rightarrow 0,
    \end{equation}
    for $n\geq 3$, since we can choose $\delta_{0}>\delta > -\frac{1}{2}$ by Proposition $\ref{prop-surjectivity}$ and Lemma $\ref{lem-delta_0}$.
    Thus, by letting $R \to +\infty$ and $\epsilon \to 0$, we obtain
    \begin{eqnarray*}    \int_{M}\left[|\nabla\varphi|^2+\frac{{\rm Scal}_{g}}{4}\varphi^2\right]d\operatorname{vol}_{g}&=&\frac{1}{4}\omega_{n}m(g).
    \end{eqnarray*}
    Therefore,  ${\rm Scal}_{g}\geq 0$ implies $m(g)\geq 0$. 
    
    If $m(g)=0$, then $\varphi$ is a parallel spinor $\nabla\varphi=0$ and hence $(M,g,o)$ is Ricci flat. 
    By Lemma $\ref{lem-delta_0}$, if $\delta_0 \neq 0$, then $\delta_0 >0$. Lemma \ref{lem-asymptotic-xi1} further implies $\varphi \to 0$ at the singular point $o$. This contradicts with that $\varphi$ is a nonzero parallel spinor (and $\varphi$ approaches a nonzero constant spinor at infinity). Therefore, $\delta_0 =0$, and so $\lambda_{-1} = \frac{1-n}{2}$. 
     Consequently, by Lemma \ref{lem-cross-section-eigenvalue-lower-bound}, we obtain that the leading order term $\phi$ in the partial asymptotic expansion of $-\xi$($=\varphi$ near the singular point) in Lemma \ref{lem-asymptotic-xi1}, is a real Killing spinor, since it realizes the lower bound in the Dirac operator eigenvalue estimate of Friedrich \cite{Friedrich-80}. Moreover, $\varphi = \xi \to \phi$ at $o$.

    Summarizing, when $m(g)=0$, each harmonic spinor $\varphi$ must be parallel, and approaches a real Killing spinor $\phi$ on $(N,g^{N})$ at the conical singularity.

    Because $\mathbb{R}^{n}$ has $2^{\left[\frac{n}{2}\right]}$ linearly independent constant spinors, we obtain $2^{\left[\frac{n}{2}\right]}$ linearly independent harmonic spinors via solving the Dirac equation $(\ref{eqn-D})$ and taking $\psi_0$ as each of these linearly independent constant spinors on $\R^n$:
    \[\varphi_{1},\cdots,\varphi_{2^{\left[\frac{n}{2}\right]}}. \]
    Then we get $2^{\left[\frac{n}{2}\right]}$ linearly independent Killing spinors:
    \[\phi_{1},\cdots,\phi_{2^{\left[\frac{n}{2}\right]}}\]
    on $(N,g^{N})$. Thus $(N,g^{N})$ must be isometric to the standard sphere $(\mathbb{S}^{n-1},g_{\mathbb{S}^{n-1}})$, by Theorem 4 in \cite{Bar-JMSP-1996}, if we assume $N$ is simply connected. Therefore $M$ is actually smooth at $o$. Moreover, because $|h|_{\bar{g}}=O(r^{\alpha})\to 0$ as $r\to 0$, $g=\bar{g}+h$ is continuous on the whole manifold including the conical part. Note here $\bar{g}=dr^2+r^2g_{\mathbb{S}^{n-1}}=g_{\mathbb{R}^{n}}$. On $[0, 1) \times N$, we take $\bar{g} = g_{\R^n}$ as a smooth background metric, and because $|\overline{\nabla} g|_{\bar{g}} = |\overline{\nabla}(\bar{g} +h)|_{\bar{g}} = | \overline{\nabla}h |_{\bar{g}} = O(r^{\alpha -1})$ as $r \rightarrow 0$, where $\overline{\nabla}$ is the Levi-Civita connection of $\bar{g}$, see Definition \ref{defn-conic-mfld}.  This implies that $g\in W^{1,n}(S^{2}T^{\ast}M|_{[0,1)\times N}, \bar{g})$. Thus, our metric $g$ satisfies the regularity assumption in \cite{Lee-LeFloch} and \cite{JSZ2022}. Thus, by Theorem 1.1 in \cite{Lee-LeFloch} or in \cite{JSZ2022}, we conclude that $(M^{n},g)\cong(\mathbb{R}^{n},g_{\mathbb{R}^{n}})$.
\end{proof}
\begin{rmrk}
{\rm
In the proof of Theorem \ref{thm-PMT}, we use Theorem 4 in \cite{Bar-JMSP-1996} to obtain that the cross section $N$ has to be the standard sphere from its number of linearly independent real Killing spinors, which is the same as that of the standard sphere. But in the dimension of $n=4k$ with $k\in \N$ (i.e. ${\rm dim}(N) = 4k-1$), Theorem 4 in \cite{Bar-JMSP-1996} says that the real projective space $\mathbb{R}P^{n-1}$ has also the same number of Killing spinors as the standard sphere, and so the number of Killing spinors cannot distinguish $\mathbb{R}P^{n-1}$ and $\mathbb{S}^{n-1}$. So we assume $\pi_1(N) = 0$ to rule out $\mathbb{R}P^{n-1}$ in these dimensions.
}
\end{rmrk}


\section{A positive mass theorem on AF manifolds with horn singularity} \label{horn} 
In this section, we derive a positive mass theorem for $n$-dimensional, $n\geq 3$, asymptotically flat manifolds $(M^n, g, o)$ with the so-called {\em $r^b$-horn} singularity at $o$ in \cite{Cheeger-80}. It is defined in the same way as in Definitions $\ref{defn-conic-mfld}$ and $\ref{defn-AF-conic-mfld}$, except changing the model singular metric in condition ${\rm (iii)}$ of Definition $\ref{defn-conic-mfld}$ by
\begin{equation}\label{eqn-exact-horn-metric}
\overline{g}_b := dr^2 + r^{2b} g^N, \ \ \text{for some} \ \ b > 0.
\end{equation}
Clearly, $b=1$ recovers the conical singularity. But the geometry and analysis of general horn singularity is very different from conical singularity. 
First we note that the scalar curvature of this model metric $\overline{g}_b$ is given by
\begin{equation}\label{eqn-horn-scalar-curvature}
\Scal_{\overline{g}_b} = \frac{\Scal_{g^N}}{r^{2b}} - \frac{b(n-1)(nb - 2)}{r^2}.
\end{equation}
Thus, in the case when $b>1$, once $\Scal_{g^N} > 0$, the scalar curvature $\Scal_{\overline{g}_b}$ will be nonnegative for sufficiently small $r>0$; however, for $b=1$ (conical case), to guarantee the nonnegativity of the scalar curvature of the model cone, $\Scal_{g^N}$ has to be greater than or equal to $(n-1)(n-2)$. 

Moreover, the Dirac operator on a $r^b$-horn is given by 
\begin{equation}\label{eqn-Dirac-horn}
\overline{D}^{\overline{g}_b} \varphi = \partial_r \cdot (\partial_r \varphi) + \frac{b(n-1)}{2r} \partial_r \cdot \varphi + \frac{1}{r^b} \partial_r \cdot D^{g^N} \varphi.
\end{equation}
This formula follows similarly as $(\ref{eqn-Dirac-op-relation})$ by using the mean curvature in Lemma $\ref{lem-boundary-inequality-model-horn}$ below. Note that the expression on the right hand side in $(\ref{eqn-Dirac-op-relation})$ has homogeneity for the scaling of the radial variable $r$, and this nice homogeneity property is crucial in the derivations of the weighted elliptic estimate in Proposition $\ref{prop-weight-elliptic-estimate}$ and the refined estimate in Proposition $\ref{prop-refined-weighted-elliptic-estimate}$. However, the expression on the right hand side of $(\ref{eqn-Dirac-horn})$ has no such homogeneity for $b \neq 1$. This causes some essential difficulties for extending the analysis in the conical case to the general horn singularity. As a result, Witten's approach may not be directly adapted to prove a positive mass theorem in the general horn singularity cases. 

Instead, 
by truncating off the singularity, 
we obtain a positive mass theorem for AF manifolds with $r^b$-horn singularity ($b >1$)
by applying the positive mass theorems of Herzlich on AF manifolds with boundary \cite{Herzlich-CMP-97} and \cite{Herzlich-02}.

On the other hand, as observed in \cite[Section 4.2]{Bray-Jauregui-AJM-2013} (see also \cite[Proposition 2.3]{Shi-Tam-PJM-2018}), the negative mass Schwarzschild manifold can be viewed as a AF manifold with $r^{\frac{2}{3}}$-horn singularity, whose scalar curvature is zero but its mass (at infinity) is negative. The mass of horn singularities and the mass of AF manifolds with horn singularities have also been studied in \cite{Bray-Jauregui-AJM-2013}.

\begin{center}
\begin{tikzpicture}
\draw  (-4,1.5) arc (180:270: 3 and 1.5);
\draw  (-1,0) arc (90:180: 3 and 1.5);
\draw [rotate around={180:(-2.1,-.1)}] (-2.1,-.1) arc (90:270:.05 and .1);
\draw [densely dashed] (-2.1,.1) arc (90:270:.05 and .1);
\draw [rotate around={180:(-3,-.38)}] (-3,-.38) arc (90:270:.1 and .38);
\draw [densely dashed] (-3,.38) arc (90:270:.1 and .38);
\node[above] (O) at (-1,0) {$O$};
\node[left] (M) at (-4,0) {$M$};
\node[above] (C) at (1.5,.2) {Remove a };
\node[above] (C) at (1.5,-.4) {neighborhood of};
\node[above] (C) at (1.5,-.8) {the singular point};
\draw [|<->|] (-3,-.8)--(-1,-.8);
\draw [->] (-.5,.2)--(3.5,.2);
\node [below] (g) at (-2,-.8) {$g=dr^2+r^{2b}g^{N}$};
\node [below] (T) at (-2.5,-1.5) {The $r^{2b}$-horn singularity};
\draw  (4,1.5) arc (180:270: 3 and 1.3);
\draw  (7,0) arc (90:180: 3 and 1.3);
\draw [rotate around={180:(7,0)}] (7,0) arc (90:270:.05 and .1);
\draw [densely dashed] (7,.2) arc (90:270:.05 and .1);
\draw [rotate around={180:(5.4,-.2)}] (5.4,-.2) arc (90:270:.1 and .3);
\draw [densely dashed] (5.4,.4) arc (90:270:.1 and .3);
\node [right] (M1) at (7,0.1) {$\tilde{M}$};
\draw [->] (7,.3)--(7,.5);
\node [above] at (7,0.4) {$r=r_{0}$};
\draw [->] (5.4,.5)--(5.4,.7);
\node [above] at (5.4,.6) {$N\sim S^{2}$};
\draw [|<->|] (5.4,-.5)--(7,-.5);
\node [below] at (6,-1.5) {The $r^{2b}$-horn boundary};
\node [below] at (6.3,-.6) {$g=dr^2+r^{2b}g^{N}$};
\node [below] at (1.5,-1.8) {Figure 1};
\end{tikzpicture}
\end{center}

We recall Herzlich's results for readers' convenience. For a closed Riemannian manifold $(N^n, g)$, its
Yamabe invariant 
is 
\begin{equation}
Y(N) = \inf_{g' \in [g]} \left\{ \frac{\int_{N} \Scal(g') d\vol_{g'}}{ \left( {\rm Vol}_{g'}(N) \right)^{\frac{n-2}{n}}} \right\},
\end{equation}
where $[g]$ is the conformal class of $g$.
\begin{thm}[Proposition 2.1 in \cite{Herzlich-02}]\label{thm-Herzlich}
Let $(M^n, g)$ be a $n$-dimensional asymptotically flat spin  manifold with boundary. Suppose that $\partial M$ has $k$ connected components $N_1, \cdots, N_k$ and $Y(N_i)\geq 0$, for all $1\leq i\leq k$, and the mean curvature $H_{N_i}$ of each $N_i$ satisfies 
\begin{equation}\label{eqn-boundary-inequality-Herzlich}
H_{N_i} \leq \left( {\rm Area}_g(N_i) \right)^{-\frac{1}{n-1}} \sqrt{\frac{n-1}{n-2} Y(N_i)}.
\end{equation}
Then, if the scalar curvature of $(M, g)$ is nonnegative, its mass is nonnegative. Moreover, if its mass is zero, then the manifold is flat.
\end{thm}

In the 3-dimensional case, the boundary is a 2-dimensional manifold. If the Yamabe invariant $Y(N_i)>0$, then the Gauss-Bonnet formula implies that $N_i$ is homeomorphic to $2$-sphere, and the condition $(\ref{eqn-boundary-inequality-Herzlich})$ becomes
\begin{equation}\label{eqn-boundary-inequality-Herzlich-3D}
H_{N_i} \leq 4 \sqrt{\frac{\pi}{{\rm Area}(N_i)}},
\end{equation}
as in \cite[Proposition 2.1]{Herzlich-CMP-97}.


\subsection{The 3-dimensional case}\label{subsect-3D-horn} 
In this subsection, we show that Herzlich's result implies  the following positive mass theorem for asymptotically flat 3-dimensional manifold with isolated $r^b$-horn singularity. For the simplicity of notation, we only consider the case of single horn singularity.
\begin{cor}\label{cor-PMT-horn}
Let $(M^3, g, o)$ be an AF spin manifold with a $r^b$-horn singularity at $o$. Assume $b>1$, and the cross section of a horn neighborhood of the singular point is diffeomorphic to the $2$-sphere. If the scalar curvature of $g$ is nonnegative, then the mass (at infinity) as in Definition $\ref{def-mass}$ is strictly positive.
\end{cor}

Corollary $\ref{cor-PMT-horn}$ is a consequence of Theorem $\ref{thm-Herzlich}$. We 
put $\widetilde{M} = M \setminus \left( (0, r_0)\times N \right)$ for some $1 > r_0 >0$. Then $(\widetilde{M}, g|_{\widetilde{M}})$ is an AF smooth Riemannian 3-dimensional manifold with boundary $\partial \widetilde{M} = \{r_0\} \times N =: N_{r_0}$, and nonnegative scalar curvature. Then it suffices to derive the inequality in $(\ref{eqn-boundary-inequality-Herzlich})$ on $N_{r_0} := \{r_0\} \times N$, for sufficiently small $r_0 >0$, with respect to the asymptotic horn metric 
\begin{equation}\label{eqn-asymptotic-horn-metric}
g = \overline{g}_b + h,
\end{equation}
where $h$ satisfies
\begin{equation}\label{eqn-asymptotic-horn-metric-control}
|\on^k h|_{\overline{g}_b} = O(r^{\alpha - k}), \ \ \text{as} \ \ r \to 0,
\end{equation}
for some $\alpha >0$ and $k = 0, 1, 2$, where $\on$ is the Levi-Civita connection of $\overline{g}_b$. Straightforward calculations on model horn $((0, 1) \times N, \overline{g}_b)$ give the following lemma.
\begin{lem}\label{lem-boundary-inequality-model-horn}
Assume $b>1$. The area ${\rm Area}(N_r, \overline{g}_b)$ of $N_r = \{r\} \times N$ with respect to the metric induced by $\overline{g}_b$ is given by
\begin{equation}
{\rm Area}(N_r, \overline{g}_b) = r^{2b} {\rm Area}(N, g^N),
\end{equation}
where ${\rm Area}(N, g^N)$ is the area of $N$ with respect to the Riemannian metric $g^N$.
The mean curvature $H_{N_r, \overline{g}_b}$ of $N_r$ with respect to $\overline{g}_b$ and the unit normal vector field $\partial_r$ is then given by
\begin{equation}
H_{N_r, \overline{g}_b} = \frac{\frac{d}{dr} {\rm Area}(N_r, \overline{g}_b)}{{\rm Area}(N_r, \overline{g})} = \frac{2b}{r}.
\end{equation}
Then for any sufficiently small $r >0$ we have
\begin{equation}
4 \sqrt{\frac{\pi}{{\rm Area}(N_{r}, \overline{g}_b)}} = 4 \sqrt{\frac{\pi}{r^{2b} {\rm Area}(N, g^N)}} = 4 \sqrt{\frac{\pi}{{\rm Area}(N, g^N)}} \frac{1}{r^b} > \frac{2b}{r} = H_{N_{r}, \overline{g}_b}.
\end{equation}
\end{lem}

Lemma $\ref{lem-boundary-inequality-model-horn}$ and the asymptotic control in (\ref{eqn-asymptotic-horn-metric-control}) imply the following lemma.
\begin{lem}
Assume $b>1$. For a sufficiently small $r_0>0$, we have
\begin{equation}
4 \sqrt{\frac{\pi}{{\rm Area}(N_{r_0}, g)}} \geq H_{N_{r_0}, g}.
\end{equation}
Here $H_{N_r, g}$ denotes the mean curvature of $N_r := \{r\}\times N$ with respective to the metric $g$ restricted on $(0, 1) \times N$ and unit normal vector field pointing to $\{1\} \times N$.
\end{lem}
\begin{proof}
By the asymptotic control hypothesis (\ref{eqn-asymptotic-horn-metric-control}) of the metric $g$, and the formulas of ${\rm Area}(N_{r}, \overline{g}_b)$ and $H_{N_{r}, \overline{g}_b}$ in Lemma \ref{lem-boundary-inequality-model-horn}, we have
 $(\ref{eqn-asymptotic-horn-metric-control})$ of the metric $g$ implies
\begin{equation*}
H_{N_r, g} = H_{N_r, \overline{g}_b} + O(r^{\alpha-1}) = \frac{2b}{r} + O(r^{\alpha -1}), \ \ \text{as} \ \ r \to 0,
\end{equation*}
and 
\begin{equation*}
{\rm Area}(N_r, g) = {\rm Area}(N_r, \overline{g}_b)  + O(r^{2b + \alpha}) = r^{2b} {\rm Area}(N, g^N) + O(r^{2b + \alpha}), \ \ \text{as} \ \ r \to 0.
\end{equation*}
Consequently, because $\alpha >0$, we obtain that for any sufficiently small $r >0$,
\begin{equation*}
H_{N_{r}, g} \leq \frac{4b}{r}, \ \ \text{and} \ \ {\rm Area}(N_{r}, g) \leq 2 r^{2b}{\rm Area}(N, g^N).
\end{equation*}
Therefore, for a sufficiently small $r_0 >0$, we have
\begin{equation*}
  4 \sqrt{\frac{\pi}{{\rm Area}(N_{r_0}, g)}} \geq 4 \sqrt{\frac{\pi}{2r^{2b}_0 {\rm Area{N, g^N}}}} 
  = \frac{4}{r^b_0} \sqrt{\frac{\pi}{2{\rm Area}(N, g^N)}} 
  > \frac{4b}{r_0} 
  \geq H_{N_{r_0}, g}.
\end{equation*}

\end{proof}
Thus 
the AF manifold with boundary $\left( \widetilde{M}, g|_{\widetilde{M}} \right)$ satisfies the boundary inequality $(\ref{eqn-boundary-inequality-Herzlich})$. As a result, Theorem $\ref{thm-Herzlich}$ implies that the mass of $\left( \widetilde{M}, g|_{\widetilde{M}} \right)$ is nonnegative. Then strict positivity of the mass follows from the rigidity result in Theorem $\ref{thm-Herzlich}$, since the asymptotic horn metric cannot be flat for $b>1$. Finally, because the mass (at infinity) of the original AF manifold $(M^3, g, o)$ with a single horn singularity is the same as the mass of $\left( \widetilde{M}, g|_{\widetilde{M}} \right)$, Corollary $\ref{cor-PMT-horn}$ then follows.


\subsection{Higher dimensional case}
In this subsection, we show that Herzlich's result, as stated in Theorem \ref{thm-Herzlich}, implies positive mass theorems for asymptotically flat manifolds with a single horn singularity for dimension $n>3$. For the simplicity of notation, we still consider the case of single horn singularity.

Let $(M^n, g, o)$ be a $n$-dimensional AF manifold with an {\em exact} horn singularity at $o$, that is, near the singular point $o$ the metric $g = \bar{g}_b$. As in \S\ref{subsect-3D-horn}, for a small $r_0 >0$, let $\widetilde{M} = M \setminus \left( (0, r_0) \times N \right)$, and consider the AF manifold $\left( \widetilde{M}, g|_{\widetilde{M}} \right)$ with boundary $\partial \widetilde{M} = \{r_0\} \times N =: N_{r_0}$. Note that by the scalar curvature formula in $(\ref{eqn-horn-scalar-curvature})$, if $\Scal_{\bar{g}_b} \geq 0$, then $\Scal_{g^N}>0$, and so the Yamabe invariant $Y(N, g^N) > 0$. Thus, we obtain that
\begin{eqnarray}
    \left( {\rm Area}_g(\partial \widetilde{M}) \right)^{-\frac{1}{n-1}} \sqrt{\frac{n-1}{n-2}Y(\partial \widetilde{M})} 
    & = & \frac{1}{r^b_0} \left( {\rm Area}_{g^N}(N) \right)^{-\frac{1}{n-1}} \sqrt{\frac{n-1}{n-2} Y(N, g^N)} \nonumber \\
    & > & \frac{(n-1)b}{r_0} \label{eqn-horn-Herzlich-condition} \\
    & = &  H_{\partial \widetilde{M}, g}, \nonumber
\end{eqnarray}
holds for sufficiently small $r_0 >0$, since $Y(N, g^N)>0$ and $b>1$. Hence Theorem \ref{thm-Herzlich} implies:
\begin{cor}\label{cor-PMT-exact-horn}
Let $(M^n, g, o)$, $n\geq 3$, be a AF spin manifold with an exact horn singularity at $o$. If the scalar curvature $\Scal_g \geq 0$, then the mass of $(M, g)$ is strictly positive.
\end{cor}

Now let $(M^n, g, o)$ be an AF manifold with a general horn singularity at $o$, that is, near the singular point $o$ the metric $g = \bar{g}_b + h$, where $h$ satisfies $(\ref{eqn-asymptotic-horn-metric-control})$. In this case, due to the presence of the perturbation term $h$ in the metric $g$, the scalar curvature of $g$ near the singular point $o$ may include some terms with orders between $\frac{1}{r^2}$ and $\frac{1}{r^{2b}}$, besides the scalar curvature of $\bar{g}_b$ as in $(\ref{eqn-horn-scalar-curvature})$. As a result, unlike the exact horn singularity case, $\Scal_g \geq 0$ only implies $\Scal_{g^N} \geq 0$, and so the Yamabe invariant $Y(N, g^N) \geq 0$. Thus, to drive a positive mass theorem, we assume $Y(N, g^N) > 0$. Then the inequality $(\ref{eqn-horn-Herzlich-condition})$ holds for sufficiently small $r_0$, and so, as in \S\ref{subsect-3D-horn}, by using the asymptotic control $(\ref{eqn-asymptotic-horn-metric-control})$, the boundary of $(\widetilde{M}, g|_{\widetilde{M}})$ satisfies the condition in $(\ref{eqn-boundary-inequality-Herzlich})$. Consequently, Theorem \ref{thm-Herzlich} implies  the positive mass theorem as following.
\begin{cor}
Let $(M^n, g, o)$, $n\geq 3$, be a AF manifold with a horn singularity at $o$. If the scalar curvature $\Scal_g \geq 0$, and the Yamabe invariant $Y(N, g^N) > 0$, then the mass of $(M, g)$ is strictly positive.
\end{cor}



\bibliographystyle{plain}
\bibliography{DSW_PMT.bib}

\end{document}